\documentclass[12pt, reqno]{amsart}
\usepackage[margin=1in]{geometry}
\usepackage{amssymb,latexsym,amsmath,amscd,amsfonts}
\usepackage{latexsym}
\usepackage[mathscr]{eucal}
\usepackage{bm}
\usepackage{tabularx}
\usepackage{mathptmx}
\usepackage{amssymb}
\usepackage{amsthm}
\usepackage{dcolumn}
\usepackage[all]{xy}
\usepackage{enumitem}
\usepackage[utf8]{inputenc}
\usepackage{mathtools}

\def \qed {\hfill \vrule height6pt width 6pt depth 0pt}
\def\textmatrix#1&#2\\#3&#4\\{\bigl({#1 \atop #3}\ {#2 \atop #4}\bigr)}
\def\dispmatrix#1&#2\\#3&#4\\{\left({#1 \atop #3}\ {#2 \atop #4}\right)}
\newcommand{\beg}{\begin{equation}}
	\newcommand{\eeg}{\end{equation}}
\newcommand{\ben}{\begin{eqnarray*}}
	\newcommand{\een}{\end{eqnarray*}}

\newcommand{\A}{\mathbb{A}}
\newcommand{\CA}{\overline{\mathbb{A}}}
\newcommand{\HS}{\mathcal{H}}
\newcommand{\KS}{\mathcal{K}}

\newcommand{\C}{\mathbb{C}}
\newcommand{\T}{\mathbb{T}}
\newcommand{\Z}{\mathbb{Z}}
\newcommand{\D}{\mathbb{D}}
\newtheorem{thm}{Theorem}[section]
\newtheorem{cor}[thm]{Corollary}

\newtheorem{prop}[thm]{Proposition}
\numberwithin{equation}{section} \theoremstyle{definition}
\newtheorem{defn}[thm]{Definition}

\newtheorem{eg}[thm]{Example}

\def\textmatrix#1&#2\\#3&#4\\{\bigl({#1 \atop #3}\ {#2 \atop #4}\bigr)}
\def\dispmatrix#1&#2\\#3&#4\\{\left({#1 \atop #3}\ {#2 \atop #4}\right)}
\begin{document}
	\title{Regular quantum annulus unitary dilation and applications}
	\author{NITIN TOMAR}

	\address[Nitin Tomar]{Mathematics Department, Indian Institute of Technology Bombay, Powai, Mumbai-400076, India.} \email{tomarnitin414@gmail.com, tnitin@math.iitb.ac.in}		
	
	
	
	\begin{abstract}
		Consider the annulus $\mathbb{A}_r=\{z \in \C: r^{-1} <|z|<r\}$ for $r>1$ and define
		\[
		Q\mathbb{A}_r=\{T:\ T \text{ is an invertible operator and } \ \|T\|, \|T^{-1}\|\leq r\}.
		\]
		The collection $Q\mathbb{A}_r$ is called the \textit{quantum annulus}. McCullough and Pascoe proved that $T \in Q\A_r$ if and only if $\beta(T^*, T)=(r^2+r^{-2})-T^*T-(T^*T)^{-1} \geq 0$. We call an invertible operator $T$ a quantum annulus unitary if $\beta(T^*, T)=0$. In this article, we present an explicit construction of a doubly commuting $d$-tuple of quantum annulus unitaries that simultaneously extends a given doubly commuting $d$-tuple of operators in $Q\mathbb{A}_r$. We introduce the notion of a regular quantum annulus unitary dilation and prove that the dilation arising from our construction is a regular one. Examples are presented to illustrate the notion of regular quantum annulus unitary dilation. As an application of the dilation theorem, we show that $\overline{\mathbb A}_r$ is a complete $K_t$-spectral set for operators in $Q\A_r$ and $\CA_r^d$ is a complete $K_{dc}^{(d)}$-spectral set for doubly commuting $d$-tuples of operators in $Q\A_r$, where
		\[
		K_t=2\left(1+\frac{2r^2}{(r^2+1)\sqrt{r^4-1}}\right) \quad \text{and} \quad 
		 K_{dc}^{(d)}=\left[2\left(1+\frac{2r^2}{(r^2+1)\sqrt{r^4-1}}\right)\right]^d.
		 \] 
		 We also prove that every doubly commuting tuple of operators in $Q\A_r$ is similar to a commuting tuple of operators having $\overline{\mathbb{A}}_r^d$ as a complete spectral set. Let $K_{dc}(\CA_r^d)$ and $K_{dc}^{\mathrm{full}}(\CA_r^d)$ be the smallest constants for which $\CA_r^d$ is a $K_{dc}(\CA_r^d)$-spectral set and a complete $K_{dc}^{\mathrm{full}}(\CA_r^d)$-spectral set, respectively, for every doubly commuting $d$-tuple of operators in $Q\A_r$. Then the resulting bounds are given by 
		 \[
		 2^d \leq K_{dc}(\CA_r^d) \leq K_{dc}^{\text{full}}(\CA_r^d) \leq K_{dc}^{(d)},
		 \] 
which further imply that  $\underset{r \to \infty}{\lim} \ K_{dc}(\CA_r^d)=\underset{r \to \infty}{\lim} \ K_{dc}^{\text{full}}(\CA_r^d)=2^d$.	En route, we provide an alternative characterization of operators in $Q\A_r$ and quantum annulus unitaries. We also prove that $\overline{\mathbb{A}}_r^d$ is a $K$-spectral set for a certain subclass of commuting $d$-tuples of operators in $Q\A_r$.

	\end{abstract}	
	
	\maketitle
	
	\noindent 
	
	\section{Introduction}\label{sec_intro}
	
	\vspace{0.1cm}
	
	\noindent Throughout the paper, all operators are bounded linear maps defined on complex Hilbert spaces. A contraction is an operator with norm at most $1$. We shall use the following notations: $\mathbb{C}$ denotes the complex plane, $\mathbb{D}$ is the open unit disc $\{z:|z|<1\}$, and $\mathbb{T}$ is the unit circle $\{z:|z|=1\}$. For a compact set $X\subseteq\mathbb{C}^d$, we denote by $\text{Rat}(X)$ the algebra of rational functions with singularities outside $X$, equipped with the supremum norm $\|\cdot\|_{\infty,X}$. For $K>0$, we say that a compact set $X\subseteq\mathbb{C}^d$ is a \textit{$K$-spectral set} for a commuting operator tuple $\underline{T}=(T_1,\dots,T_d)$ if its Taylor joint spectrum $\sigma_T(\underline{T})$ is contained in $X$ and for all $f\in\text{Rat}(X)$,
	\begin{align}\label{eqn_spec}
		\|f(\underline{T})\|\leq K\|f\|_{\infty,X}.
	\end{align}
	Such a constant $K$ is called a \textit{spectral constant}. If \eqref{eqn_spec} holds for all matricial functions, that is, 
	\[
	\|[f_{ij}(\underline{T})]_{i, j =1}^n\|\leq K\|[f_{ij}]_{i, j=1}^n\|_{\infty, X}=K \cdot \sup\left\{ \|[f_{ij}(x)]_{i, j=1}^n\| : x \in X \right\}
	\] 
	for all $n \in \mathbb{N}$ and $[f_{ij}]$ in $M_n(\text{Rat}(X))$, then $X$ is called a \textit{complete $K$-spectral set} for $\underline{T}$. The set $X$ is called a \textit{spectral set} (respectively, a \textit{complete spectral set}) for $\underline{T}$ if it is a $K$-spectral set (respectively, a complete $K$-spectral set) with $K=1$. For $r>1$, consider the annulus 
	\[
	\mathbb{A}_r=\{z \in \mathbb{C}: r^{-1}<|z|<r\}.
	\] 
Let $Q\mathbb{A}_r$ be the collection of all invertible operators $T$ that satisfies $\|T\|, \|T^{-1}\| \leq r$. The class $Q\A_r$ is known as the \textit{quantum annulus}. A fundamental characterization of operators in $Q\mathbb{A}_r$ was obtained by McCullough and Pascoe in \cite{Pas-McCull}. For an invertible operator $T$, let us define
\[
\beta(T^*, T):= (r^2+r^{-2})-T^*T-(T^*T)^{-1}.
\] 
It was proved in \cite{Pas-McCull} that an invertible operator $T \in Q\mathbb{A}_r$ if and only if $\beta(T^*, T)\geq 0$. They further established a dilation theorem for operators in $Q\A_r$, which states that for every $T \in Q\A_r$ acting on a Hilbert space $\HS$, there exists an invertible operator $J$ acting on a Hilbert space $\KS \supseteq \HS$ such that $\beta(J^*, J)=0$ and $T^n=P_{\HS}J^n|_{\HS}$ for every $n \in \Z$. Here, $P_\HS$ denotes the orthogonal projection of $\KS$ onto $\HS$. An invertible operator $J$ satisfying $\beta(J^*, J) = 0$ is referred to as a \textit{quantum annulus unitary}. An explicit construction of such a dilation was obtained in \cite{Pal_Pascoe}.

\begin{thm}[\cite{Pal_Pascoe}, Theorem 2.1]\label{thm_01}
	Every operator $T \in Q\A_r$ extends to a quantum annulus unitary 
	\[
	J=\begin{bmatrix}
		T & T(T^*T)^{-1\slash 2}\beta(T^*, T)^{1\slash 2}\\
		0 & T^{-*}
	\end{bmatrix}.
	\]
\end{thm}

The dilation space in Theorem \ref{thm_01} is simply $\HS \oplus \HS$. The dilation obtained in \cite{Pas-McCull} was different. If $\sigma((T^*T)^{1/2})$ is a finite subset of the open interval $(r^{-1},r)$, the dilation acts on $L^2(\mathbb{T})\otimes \HS$, whereas in the general case the dilation space is obtained through Stinespring's dilation theorem. Also, a dilation theorem for doubly commuting tuples of operators in $Q\A_r$ was proved in \cite{Tomar}. Recall that a commuting tuple of invertible operators $(T_1, \dotsc, T_d)$ on a Hilbert space $\HS$ is said to admit a dilation to a tuple of  invertible operators $(J_1, \dotsc, J_d)$  on a Hilbert space $\KS \supseteq \HS$ if
\[
T_1^{m_1}\dotsc T_d^{m_d}=P_\HS J_1^{m_1}\dotsc J_d^{m_d}|_\HS
\] 
for all integers $m_1, \dotsc, m_d$. We now state the following multivariable dilation theorem from \cite{Tomar}.

\begin{thm}[\cite{Tomar}, Theorem 1.3]\label{thm_mainQAr}
	Let $\underline{T}=(T_1, \dotsc, T_d)$ be a doubly commuting tuple of invertible operators acting on a Hilbert space $\HS$. Then $T_1, \dotsc, T_d \in Q\A_r$ if and only if $\underline{T}$ admits a dilation to a doubly commuting $d$-tuple of quantum annulus unitaries.
\end{thm} 	  

While Theorem \ref{thm_mainQAr} guarantees the existence of a dilation to a doubly commuting tuple of quantum annulus unitaries, the proof does not provide an explicit construction of the dilating tuple. Indeed, the dilation obtained in \cite{Tomar} acts on the space $L^2(\T)\otimes \HS$ when each $\sigma((T_j^*T_j)^{1/2})$ is a finite subset of the open interval $(r^{-1},r)$ and in the general case, the dilation space is obtained using Stinespring's dilation theorem. It is therefore natural to ask whether the dilation in Theorem \ref{thm_mainQAr} admits an explicit construction analogous to the one-variable construction in Theorem \ref{thm_01}. One of the main objectives of the present article is to answer this question. In fact, we refine Theorem \ref{thm_mainQAr} in several significant ways in Section \ref{sec_03}. First, we provide an explicit construction of a doubly commuting tuple of quantum annulus unitaries dilating a given doubly commuting tuple $(T_1, \dotsc, T_d)$ of operators in $Q\A_r$ acting on a Hilbert space $\HS$. Second, the resulting dilation acts on the space
\[
\underbrace{\HS \oplus \dotsb \oplus \HS}_{2^d \text{-times}},
\]
which is substantially smaller and more concrete than the dilation space arising in \cite{Tomar}. Furthermore, we show that every doubly commuting tuple of operators in $Q\A_r$ admits a simultaneous extension, and hence a dilation, to a doubly commuting tuple of quantum annulus unitaries. In fact, our construction gives more than a simultaneous extension. To explain this point, we recall the classical notion of regular unitary dilation of doubly commuting contractions from \cite{Brehmer, NagyFoias6}. Given a commuting tuple $\underline{T}=(T_1,\dots, T_d)$ of contractions acting on a Hilbert space $\HS$, we say that $\underline{T}$ admits a \textit{regular unitary dilation} if there exists a commuting tuple $(U_1,\dots,U_d)$ of unitaries on a Hilbert space $\KS \supseteq \HS$ such that
\begin{equation}\label{eqn_1001}
	P_{\HS}
	\left(
	\prod_{i\in S} U_i^{*m_i}
	\prod_{i\notin S} U_i^{m_i}
	\right)\Big|_{\HS}
	=
	\prod_{i\in S} T_i^{*m_i}
	\prod_{i\notin S} T_i^{m_i}
\end{equation}
for every subset $S\subseteq \{1,\dotsc,d\}$ and non-negative integers $m_1,\dotsc,m_d$. It is well-known that every doubly commuting tuple of contractions admits a regular unitary dilation, see, e.g., \cite{NagyFoias6} for further details. For operators in $Q\mathbb{A}_r$, it is natural to allow arbitrary integers in \eqref{eqn_1001} instead of only non-negative integers. Our construction shows that the dilating tuple of quantum annulus unitaries for doubly commuting tuple of operators in $Q\A_r$ satisfies the equality as in \eqref{eqn_1001} for integers $m_1, \dotsc, m_d$. We call such a dilation a regular quantum annulus unitary dilation. 

\smallskip 

Dilation theory has played an important role in the study of spectral constant estimates for operators in $Q\A_r$. This connection is particularly evident in the work of \cite{Pal_Pascoe}, where the dilation theorem was used to obtain an alternative proof of spectral constant estimate from \cite{Pascoe}. Motivated by this interplay, we study the classical problem of determining optimal spectral constant
\[
K(\CA_r)=\inf\{K:\ \CA_r \text{ is a } K\text{-spectral set for every } T\in Q\A_r\}.
\] 
This problem has attracted considerable attention, e.g., see \cite{Badea, CrouI, CrouII, Pal_Pascoe, Pascoe, Paulsen, Shields, TsikalasII}
and the references therein. The best known lower bound is due to Tsikalas \cite{TsikalasII} who proved that $K(\CA_r) \geq 2$, improving upon earlier bounds obtained in \cite{Badea}. On the other hand, Crouzeix and Greenbaum \cite{CrouII} showed that $K(\CA_r)\le 1+\sqrt{2}$. Pascoe \cite{Pascoe} proved that $	K(\CA_r)\leq K_p$, where
\begin{equation}\label{eqn_ub}
K_p=
	2\left[1+2r^2(r^4-1)^{-1}\right].
\end{equation}
An alternative proof of this estimate was subsequently provided by the authors of \cite{Pal_Pascoe}. Note that the bound $K_p$ is strictly smaller than $1+\sqrt{2}$ for $r \geq 4$. The estimate in \eqref{eqn_ub} is also important from an asymptotic viewpoint as it converges to $2$ as $r\to \infty$, matching the lower bound by Tsikalas \cite{TsikalasII}. Thus, $K(\CA_r) \to 2$ as $r \to \infty$. Recently, Crouzeix \cite{CrouIII} obtained a sharper upper bound for finite-dimensional quantum annulus operators satisfying $\|T\|,\|T^{-1}\|<r$. Motivated by these results, one may ask for analogous estimates in the multivariable setting. The spectral set estimates for doubly commuting tuples of operators in $Q\A_r$ were studied by the authors of \cite{Pal_Pascoe}. They proved that $\CA_r^d$ is a $K_{dc}$-spectral set for doubly commuting $d$-tuples of operators in $Q\A_r$, where
\[
K_{dc}=\left[(3r^2-1)\slash (r^2-1)\right]^d.
\]
Moreover, if $K_{dc}(\CA_r^d)$ denotes the smallest constant for which $\CA_r^d$ is a $K_{dc}(\CA_r^d)$-spectral set for all doubly commuting $d$-tuples in $Q\A_r$, then
$2^d \leq K_{dc}(\CA_r^d)\leq K_{dc}$ and $2^d\leq
\underset{r\to\infty}{\lim}K_{dc}(\CA_r^d)
\leq 3^d$. Although these estimates establish the existence of a uniform spectral constant in the doubly commuting setting, the upper bound is considerably larger than its one-variable counterpart. Indeed,
\[
\lim_{r \to \infty}\left[(3r^2-1)\slash (r^2-1)\right]=3 \quad \text{and}
\quad 
\lim_{r\to \infty}	2\left[1+2r^2(r^4-1)^{-1}\right]=2.
\]
This suggests that the bound $\left[(3r^2-1)\slash (r^2-1)\right]^d$ is not optimal and naturally leads to the problem of determining sharper multivariable spectral constant estimates. The alternative proof of Pascoe's estimate in \cite{Pal_Pascoe} capitalizes on the dilation theorem for $Q\A_r$. Thus, it is natural to investigate whether the dilation theorem for doubly commuting tuples of quantum annulus operators can be used to obtain sharper spectral constant estimates in the multivariable setting. To begin with, we refine in Section \ref{sec_02} the spectral constant $K_p$ as in \eqref{eqn_ub} for operators in $Q\A_r$. We show that the constant $K_p$ can be improved to  
\begin{equation}\label{eqn_Kt}
K_t = 2\left(1+\frac{2r^2}{(r^2+1)\sqrt{r^4-1}}\right).
\end{equation}
Clearly, $K_t \lneq K_p$ for all $r>1$. Moreover, we prove that $\CA_r$ is a complete $K_t$-spectral set for every operator $T\in Q\A_r$. Let $K(\CA_r)$ and $K^{\text{full}}(\CA_r)$ be the smallest constants such that $\CA_r$ is a $K(\CA_r)$-spectral set and a complete $K^{\text{full}}(\CA_r)$-spectral set, respectively, for every operator in $Q\A_r$. Then 
\[
2 \leq K(\CA_r) \leq K^{\text{full}}(\CA_r) \leq K_t \quad \text{and} \quad \lim_{r \to \infty}K(\CA_r)=\lim_{r \to \infty}K^{\text{full}}(\CA_r)=2.
\]
Also, we provide an alternative characterization for operators in $Q\A_r$. For an invertible operator $T$, we consider $\Phi_r(T)=r(1+r^2)^{-1}(T+T^{-*})$. We prove that $T \in Q\A_r$ if and only if $\|\Phi_r(T)\| \leq 1$, and $T$ is a quantum annulus unitary if and only if $\Phi_r(T)$ is a unitary. This gives a relationship between $Q\A_r$ and the class of contractions. We present in Section \ref{sec_04} estimates for the optimal spectral and complete spectral constants associated with doubly commuting tuples of operators in $Q\A_r$. We show that $\CA_r^d$ is a complete $K_{dc}^{(d)}$-spectral set for such tuples, where
\[
K_{dc}^{(d)}=\left[2\left(1+\frac{2r^2}{(r^2+1)\sqrt{r^4-1}}\right)\right]^d.
\]
Moreover, we prove in Theorem \ref{thm_sim} that every doubly commuting tuple of operators in $Q\A_r$ is similar to a commuting tuple of operators having $\CA_r^d$ as a complete spectral set. Let $K_{dc}(\CA_r^d)$ and $K_{dc}^{\mathrm{full}}(\CA_r^d)$ be the smallest constants such that $\CA_r^d$ is a $K_{dc}(\CA_r^d)$-spectral set and a complete $K_{dc}^{\mathrm{full}}(\CA_r^d)$-spectral set, respectively, for doubly commuting $d$-tuples of operators in $Q\A_r$. For these optimal constants, it is proved that $2^d \le K_{dc}(\CA_r^d) \le K_{dc}^{\mathrm{full}}(\CA_r^d)
	\leq K_{dc}^{(d)}$, which in turn implies that
\[
	\underset{r\to\infty}{\lim} K_{dc}(\CA_r^d)
	=
	\underset{r\to\infty}{\lim} K_{dc}^{\mathrm{full}}(\CA_r^d)
	=
	2^d.
	\]
In Section \ref{sec_05}, we study commuting tuples of operators in $Q\A_r$. Unlike the doubly commuting case, the methods developed in Section \ref{sec_04} do not directly apply to arbitrary commuting tuples. Moreover, even in the classical setting of commuting contractions, it remains unknown whether von Neumann's inequality for commuting triples holds up to a universal constant, e.g., see \cite{Hartz, Paulsen}. This leads us to consider the class  $Q\A_r^{\text{von}}$ consisting of all commuting tuples $(T_1, \dotsc, T_d)$ of operators in $Q\A_r$ such that 
\[
\|f(T_1^{\mu(1)}, \dotsc, T_d^{\mu(d)})\| \leq \|f\|_{\infty, (r\overline{\D})^d}
\]
for every holomorphic function $f$ on $(r\overline{\D})^d$ and every map $\mu: \{1, \dotsc, d\} \to \{1, -1\}$. We prove in Theorem \ref{thm_502} that $\CA_r^d$ is a $K(r, d)$-spectral set for $d$-tuples in $Q\A_r^{\text{von}}$, where 
\[
K(r, d)=\left(2+\frac{1+r^2}{\sqrt{r^4-1}}\right)^d.
\]

\section{Regular dilation of doubly commuting operators in $Q\A_r$}\label{sec_03}
	
	\vspace{0.1cm}
	
	\noindent In this section, we introduce the notion of regular quantum annulus unitary dilation and explicitly construct such a dilation for doubly commuting tuples of operators in $Q\mathbb{A}_r$.
	
	\begin{defn} 
		Let $\underline{T}=(T_1,\dotsc,T_d)$ be a commuting tuple of
		operators in $Q\mathbb{A}_r$ acting on a Hilbert space $\HS$.
		We say that $\underline{T}$ admits a \textit{regular quantum annulus unitary
			dilation} if there exist a Hilbert space $\KS \supseteq \HS$
		and a commuting tuple
		$(J_1,\dotsc,J_d)$ of quantum annulus unitaries
		on $\KS$ such that
		\[
		P_{\HS}
		\left(
		\prod_{i\in S} J_i^{*m_i}
		\prod_{i\notin S} J_i^{m_i}
		\right)\Big|_{\HS}
		=
		\prod_{i\in S} T_i^{*m_i}
		\prod_{i\notin S} T_i^{m_i}
		\]
		for every subset $S\subseteq \{1,\dotsc,d\}$ and every
		$m_1,\dotsc,m_d\in\mathbb{Z}$. 
	\end{defn} 
	
	We now present an example of a commuting pair (which is not doubly commuting) of operators in $Q\A_r$ that admits a regular quantum annulus unitary dilation. The primary idea of this example came from the $L^p$ and $H^p$ spaces of annulus that appeared in Sarason's seminal work \cite{SarasonI}. 
	
	\begin{eg}\label{eg1}
		Consider the space
		$
		L^2(\partial\mathbb{A}_r)=L^2(r^{-1}\mathbb{T})\oplus L^2(r\mathbb{T}),
		$
		where $L^2(r^{-1}\mathbb{T})$ and $L^2(r\mathbb{T})$ are endowed with normalized Lebesgue measure and the norm on this space is given by
		\begin{equation*}
			\|f\|^2=\frac{1}{2\pi}\overset{2\pi}{\underset{0}{\int}}|f(r^{-1}e^{it})|^2dt+\frac{1}{2\pi}\overset{2\pi}{\underset{0}{\int}}|f(re^{it})|^2dt.    
		\end{equation*}
		For $0 < \alpha < 1$ and $n \in \Z$, define a map $w_{\alpha+n}$ on $\partial \mathbb{A}_r$ as
		$
		w_{\alpha+n}(r^{-1}e^{it})=r^{-(\alpha+n)}e^{i(\alpha+n) t}$ and $  w_{\alpha+n}(re^{it})=r^{\alpha+n} e^{i(\alpha+n) t}$ with $0< t \leq 2 \pi$. It is easy to see that $\{w_{\alpha+n} :n \in \Z \}$ consists of mutually orthogonal maps in $L^2(\partial \mathbb{A}_r)$. Fix $\alpha \in (0, 1)$. Define a closed subspace of $L^2(\partial \A_r)$ as
		\[
		H_{\alpha}^2(\partial \mathbb{A}_r)=\overline{\text{span}}\{w_{\alpha+n} \ : \ n \in \Z\}.
		\] 
		Consider the map $A:H_{\alpha}^2(\partial \mathbb{A}_r) \to H_{\alpha}^2(\partial \mathbb{A}_r)$ given by $Aw_{\alpha+n}=w_{\alpha+n+1}$. Clearly, $A$ is invertible and $A^{-1}w_{\alpha+n}=w_{\alpha+n-1}$. Note that
		$\langle A^*(w_{\alpha+n}), w_{\alpha+m}\rangle =\langle w_{\alpha+n}, Aw_{\alpha+m}\rangle
		=\langle w_{\alpha+n}, w_{\alpha+m+1}\rangle$. Hence,
		\[ 
		\langle A^*(w_{\alpha+n}), w_{\alpha+m}\rangle= \left\{
		\begin{array}{ll}
			0 & n \ne m+1 \\
			r^{-2(\alpha+n)}+r^{2(\alpha+n)} & n=m+1 \; ,\\
		\end{array} 
		\right. 
		\]    
		which implies that $A^*(w_{\alpha+n})$ is orthogonal to $\overline{\text{span}}\{w_{\alpha+m} : m \in \mathbb{Z}, \ n\ne m+1\}$. Therefore, $A^*(w_{\alpha+n})$ is in $\mbox{span}\{w_{\alpha+n-1}\}$ and so,   $A^*(w_{\alpha+n})=r(\alpha;n)w_{\alpha+n-1}$ for some $r(\alpha;n) \in \C$. Note that
		\[
		r^{-2(\alpha+n)}+r^{2(\alpha+n)}= \langle A^*(w_{\alpha+n}), w_{\alpha+n-1}\rangle = r(\alpha;n)\|w_{\alpha+n-1}\|^2
		=r(\alpha;n)(r^{-2(\alpha+n-1)}+r^{2(\alpha+n-1)})
		\]	
		and thus, 
		\begin{equation*}
			r(\alpha;n)= \frac{r^{-2(\alpha+n)}+r^{2(\alpha+n)}}{r^{-2(\alpha+n-1)}+r^{2(\alpha+n-1)}}
		\end{equation*}
		for every $n \in \Z$. Also, $A$ is not a normal operator. To see this, note that $AA^*w_{\alpha+n}=r(\alpha; n)w_{\alpha+n}$ and $A^*Aw_{\alpha+n}=r(\alpha; n+1)w_{\alpha+n}$ for every $n \in \Z$. For $A$ to become normal, it is necessary and sufficient that $r(\alpha; n)=r(\alpha; n+1)$ for every $n \in \Z$, which is not possible since $r>1$. We now construct a quantum annulus unitary dilation of $A$, which in particular shows that $A \in Q\A_r$. Consider an operator $S: L^2(\partial\mathbb{A}_r) \to L^2(\partial\mathbb{A}_r)$ defined by
		$
		S(f)(z)=zf(z).
		$
		Then 
		\begin{align*}
			Sf&=g_1 \oplus g_2, \quad \text{where} \quad g_1(r^{-1}e^{it})=r^{-1}e^{it}f_1(r^{-1}e^{it}) \quad \text{and}  \quad g_2(re^{it})=re^{it}f_2(re^{it});\\
			S^{*}f&=k_1 \oplus k_2, \ \quad \text{where} \quad  k_1(r^{-1}e^{it})=r^{-1}e^{-it}f_1(r^{-1}e^{it}) \ \   \text{and} \quad  k_2(re^{it})=re^{-it}f_2(re^{it})
		\end{align*}
		for every $f=f_1\oplus f_2 \in L^2(r^{-1}\mathbb{T})\oplus L^2(r\mathbb{T})$. Note that $S$ is a quantum annulus unitary. To see this, let
		$f=f_1\oplus f_2 \in L^2(r^{-1}\mathbb{T})\oplus L^2(r\mathbb{T})$. Then
		\begin{align*}
			\beta(S^*, S)f
			&=(r^2+r^{-2})f-S^*Sf-(S^*S)^{-1}f\\
			&=
			(r^2+r^{-2})I_{L^2(\partial \A_r)}
			-
			\left(
			r^{-2}I_{L^2(r^{-1}\mathbb{T})}
			\oplus
			r^2I_{L^2(r\mathbb{T})}
			\right)-
			\left(
			r^2I_{L^2(r^{-1}\mathbb{T})}
			\oplus
			r^{-2}I_{L^2(r\mathbb{T})}
			\right)\\
			&=0.
		\end{align*}
		Clearly, $Sw_\alpha=w_{\alpha+1}$ and $S^{-1}w_\alpha=w_{\alpha-1}$. Consequently, $S^m=A^m|_{H^2_\alpha(\partial \A_r)}$ for every $m \in \Z$. Putting everything together, $A \in Q\A_r$ and $S$ is a quantum annulus unitary dilation of $A$. Define a commuting pair $(T_1, T_2)=(A, -A)$ of operators in $Q\A_r$ acting on $H^2_\alpha(\partial \A_r)$. Since $A$ is not normal, $(T_1, T_2)$ is not doubly commuting. Evidently, $(T_1, T_2)$ admits simultaneous extension, hence dilation, to the commuting pair $(J_1, J_2)=(S, -S)$ of quantum annulus unitaries acting on $L^2(\partial \A_r)$. Furthermore, $(J_1, J_2)$ is a regular quantum annulus unitary dilation of $(T_1, T_2)$. \qed 
	\end{eg}
	
	Next, we give an example of a commuting pair of operators in $Q\mathbb{A}_r$ which admits dilation to a commuting pair of quantum annulus unitaries that is not regular.
	
	\begin{eg}
		We continue with the notation of Example \ref{eg1}. Fix $\alpha \in (0, 1)$. Consider the commuting pairs $(M_1, M_2)=(A^*, -A^*)$ and $(N_1, N_2)=(S^*, -S^*)$ of operators in $Q\A_r$ acting on $H^2_\alpha(\partial \A_r)$ and $L^2(\partial \A_r)$, respectively. Since $\beta(S^*, S)=0$, it follows that
		\[
		\beta(S, S^*)=(r^2+r^{-2})-SS^*-(SS^*)^{-1}=S\left((r^2+r^{-2})-S^*S-(S^*S)^{-1}\right)S^{-1}=S\beta(S^*, S)S^{-1}=0
		\]
		and so, $(N_1, N_2)$ is a commuting pair of quantum annulus unitaries. By Example \ref{eg1}, we have that
		\[
		P_{H^2_\alpha(\partial \A_r)}N_1^{n}N_2^{m}|_{H^2_\alpha(\partial \A_r)}=(-1)^mP_{H^2_\alpha(\partial \A_r)}(S^{n+m})^{*}|_{H^2_\alpha(\partial \A_r)}=(-1)^m(A^{n+m})^*=M_1^nM_2^m.
		\]
		Thus, $(N_1, N_2)$ is a quantum annulus unitary dilation of $(M_1, M_2)$. We show that $(N_1, N_2)$ is not a regular quantum annulus unitary dilation of $(M_1, M_2)$. By Example \ref{eg1}, $A^*(w_{\alpha+n})=r(\alpha;n)w_{\alpha+n-1}$, where 
		\[
		r(\alpha;n)= \frac{r^{-2(\alpha+n)}+r^{2(\alpha+n)}}{r^{-2(\alpha+n-1)}+r^{2(\alpha+n-1)}}
		\]
		for every $n \in \Z$. Suppose $(N_1, N_2)$ is a regular quantum annulus unitary dilation of $(M_1, M_2)$. Then 
		\begin{equation}\label{eqn_2002}
			AA^*=-M_1^*M_2=-P_{H^2_\alpha(\partial \A_r)}N_1^*N_2|_{H^2_\alpha(\partial \A_r)}=P_{H^2_\alpha(\partial \A_r)}SS^*|_{H^2_\alpha(\partial \A_r)}=P_{H^2_\alpha(\partial \A_r)}S^*S|_{H^2_\alpha(\partial \A_r)},
		\end{equation}
		where the last equality holds since $S$ is a normal operator. The normality of $S$ follows from the fact that $S^*S=r^{-2}I_{L^2(r^{-1}\mathbb{T})}
		\oplus
		r^2I_{L^2(r\mathbb{T})}=SS^*$. Evidently, $AA^*w_{\alpha+n}=r(\alpha; n)w_{\alpha+n}$ for every $n \in \Z$ and   
		\[
		P_{H^2_\alpha(\partial \A_r)}S^*Sw_{\alpha+n}=\sum_{k \in \Z} \frac{\langle S^*Sw_{\alpha+n}, w_{\alpha+k}\rangle}{\|w_{\alpha+k}\|^2}w_{\alpha+k}=\sum_{k \in \Z} \frac{\langle w_{\alpha+n+1}, w_{\alpha+k+1}\rangle}{\|w_{\alpha+k}\|^2}w_{\alpha+k}=\frac{\|w_{\alpha+n+1}\|^2}{\|w_{\alpha+n}\|^2}w_{\alpha+n}.
		\]
		It follows from \eqref{eqn_2002} that $\displaystyle r(\alpha; 0)w_{\alpha}=AA^*w_{\alpha}=P_{H^2_\alpha(\partial \A_r)}S^*Sw_\alpha=\frac{\|w_{\alpha+1}\|^2}{\|w_{\alpha}\|^2}w_{\alpha}$ and we have
		\[
		\frac{r^{-2\alpha}+r^{2\alpha}}{r^{-2(\alpha-1)}+r^{2(\alpha-1)}}=\frac{r^{-2(\alpha+1)}+r^{2(\alpha+1)}}{r^{-2\alpha}+r^{2\alpha}}.
		\]
	A straightforward calculation shows that the above equality holds if and only if $r=1$, which is a contradiction. Thus, $(N_1, N_2)$ is not a regular quantum annulus unitary dilation of $(M_1, M_2)$. \qed 
	\end{eg}
	
	For clarity, we first prove the dilation theorem for doubly commuting pairs of operators in $Q\mathbb{A}_r$.

	\begin{thm}\label{thm_02}
		Let $\underline{T}=(T_1, T_2)$ be a doubly commuting pair of operators in $Q\A_r$ acting on a Hilbert space $\HS$. Consider the operators $J_1, J_2$ on $\HS \oplus \HS \oplus \HS \oplus \HS$ given by
		\[
		J_1= \begin{bmatrix}
			T_1 & 0 & D_1 & 0 \\
			0 & T_1 & 0 & D_1 \\
			0 & 0 & T_1^{-*} & 0 \\
			0 & 0 & 0 & T_1^{-*}
		\end{bmatrix} \qquad \text{and} \qquad J_2= \begin{bmatrix}
			T_2 & D_2 & 0 & 0 \\
			0 & T_2^{-*} & 0 & 0 \\
			0 & 0 & T_2 & D_2 \\
			0 & 0 & 0 & T_2^{-*}
		\end{bmatrix},
		\]
		where $D_i=T_i(T_i^*T_i)^{-1\slash 2}\beta(T_i^*, T_i)^{1\slash 2}$ for $i=1, 2$. Then $(J_1, J_2)$ is a doubly commuting pair of quantum annulus unitaries such that $(J_1, J_2)$ simultaneously extends $(T_1, T_2)$, that is, 
		\[
		T_1^{n_1}T_2^{n_2}=J_1^{n_1}J_2^{n_2}|_\HS \quad \text{for all $n_1, n_2 \in \Z$}.
		\] 
Moreover, $(J_1, J_2)$ is a regular quantum annulus unitary dilation of $(T_1, T_2)$.
	\end{thm}	
	
	\begin{proof}
		Since $T_1, T_2$ are doubly commuting operators, it follows that $T_i(T_j^*T_j)=(T_j^*T_j)T_i$ for $1 \leq i, j \leq 2$ with $i \ne j$. An application of spectral theorem gives the commutativity of $\beta(T_j^*, T_j)^{1\slash 2}$ with $T_i$ and $(T_i^*T_i)^{1\slash 2}$ for $i \ne j$. Consequently, it follows that
		\begin{equation}\label{eqn_2003}
			T_iD_j=D_jT_i, \quad T_iD_j^*=D_j^*T_i, \quad D_iD_j=D_jD_i \quad \text{and} \quad D_iD_j^*=D_j^*D_i
		\end{equation} 
		for $1 \leq i, j \leq 2 $ with $i \ne j$. A routine computation shows that
		\[
		J_1J_2=\begin{bmatrix}
			T_1T_2 & T_1D_2 & D_1T_2 & D_1D_2\\
			0 & T_1T_2^{-*} & 0 & D_1T_2^{-*}\\
			0 & 0 & T_1^{-*}T_2 & T_1^{-*}D_2\\
			0 & 0 & 0 & T_1^{-*}T_2^{-*}
		\end{bmatrix} \quad \text{and} \quad J_2J_1=\begin{bmatrix}
			T_2T_1 & D_2T_1 & T_2D_1 & D_2D_1\\
			0 & T_2^{-*}T_1 & 0 & T_2^{-*}D_1\\
			0 & 0 & T_2T_1^{-*} & D_2T_1^{-*}\\
			0 & 0 & 0 & T_2^{-*}T_1^{-*}
		\end{bmatrix}.
		\]	
		Furthermore,
		\[
		J_1J_2^*=\begin{bmatrix}
			T_1T_2^* & 0 & D_1T_2^* & 0\\
			T_1D_2^* & T_1T_2^{-1} & D_1D_2^* & D_1T_2^{-1}\\
			0 & 0 & T_1^{-*}T_2^* & 0\\
			0 & 0 & T_1^{-*}D_2^* & T_1^{-*}T_2^{-1}
		\end{bmatrix} \quad \text{and} \quad J_2^*J_1=\begin{bmatrix}
			T_2^*T_1 & 0 & T_2^*D_1 & 0\\
			D_2^*T_1 & T_2^{-1}T_1 & D_2^*D_1 & T_2^{-1}D_1\\
			0 & 0 & T_2^*T_1^{-*} & 0\\
			0 & 0 & D_2^*T_1^{-*} & T_2^{-1}T_1^{-*}
		\end{bmatrix}.
		\]	
		Consequently, it follows from \eqref{eqn_2003} that $(J_1, J_2)$ is a doubly commuting pair. Moreover,
		\[
		J_1^{-1}=\begin{bmatrix}
			T_1^{-1} & 0 & -D_1^* & 0 \\
			0 & T_1^{-1} & 0 & -D_1^* \\
			0 & 0 & T_1^{*} & 0 \\
			0 & 0 & 0 & T_1^{*}
		\end{bmatrix} \quad \text{and} \quad J_2^{-1}=\begin{bmatrix}
			T_2^{-1} & -D_2^* & 0 & 0 \\
			0 & T_2^{*} & 0 & 0 \\
			0 & 0 & T_2^{-1} & -D_2^* \\
			0 & 0 & 0 & T_2^{*}
		\end{bmatrix}.
		\]
		A straightforward calculation shows that
		{\small 
			\begin{align*}
				J_1^*J_1
				&=\begin{bmatrix}
					T_1^*T_1 & 0 & T_1^*D_1 & 0\\
					0 & T_1^*T_1 & 0 & T_1^*D_1\\
					D_1^*T_1 & 0 & D_1^*D_1+(T_1^*T_1)^{-1} & 0\\
					0 & D_1^*T_1 & 0 & D_1^*D_1+(T_1^*T_1)^{-1}
				\end{bmatrix}\\
				&=\begin{bmatrix}
					T_1^*T_1 & 0 & (T_1^*T_1)^{1\slash 2}\beta(T_1^*, T_1)^{1\slash 2} & 0\\
					0 & T_1^*T_1 & 0 & (T_1^*T_1)^{1\slash 2}\beta(T_1^*, T_1)^{1\slash 2}\\
					\beta(T_1^*, T_1)^{1\slash 2}(T_1^*T_1)^{1\slash 2} & 0 & (r^2+r^{-2})-T_1^*T_1 & 0\\
					0 & \beta(T_1^*, T_1)^{1\slash 2}(T_1^*T_1)^{1\slash 2} & 0 & (r^2+r^{-2})-T_1^*T_1
				\end{bmatrix},
			\end{align*}
		}
		where the last equality follows from the facts that $T_1^*D_1=(T_1^*T_1)^{1\slash 2}\beta(T_1^*, T_1)^{1\slash 2}$ and $D_1^*D_1+(T_1^*T_1)^{-1}=(r^2+r^{-2})-T_1^*T_1$. Again, a simple calculation shows that $(J_1^*J_1)^{-1}$ equals
		{\small 
			\begin{align*}
				\begin{bmatrix}
					(r^2+r^{-2})-T_1^*T_1 & 0 & -\beta(T_1^*, T_1)^{1\slash 2}(T_1^*T_1)^{1\slash 2} & 0\\
					0 & (r^2+r^{-2})-T_1^*T_1 & 0 & -\beta(T_1^*, T_1)^{1\slash 2}(T_1^*T_1)^{1\slash 2}\\
					-(T_1^*T_1)^{1\slash 2}\beta(T_1^*, T_1)^{1\slash 2} & 0 & T_1^*T_1 & 0\\
					0 & -(T_1^*T_1)^{1\slash 2}\beta(T_1^*, T_1)^{1\slash 2} & 0 & T_1^*T_1
				\end{bmatrix}
			\end{align*}
		}
		and so, $J_1^*J_1+(J_1^*J_1)^{-1}=(r^2+r^{-2})I$. In a similar fashion, we have
		{\small 
			\begin{align*}
				J_2^*J_2&=\begin{bmatrix}
					T_2^*T_2 & T_2^*D_2 & 0 & 0 \\
					D_2^*T_2 & D_2^*D_2+(T_2^*T_2)^{-1} & 0 & 0 \\
					0 & 0 & T_2^*T_2 & T_2^*D_2 \\
					0 & 0 & D_2^*T_2 & D_2^*D_2+(T_2^*T_2)^{-1}
				\end{bmatrix}\\
				&=\begin{bmatrix}
					T_2^*T_2 & (T_2^*T_2)^{1\slash 2}\beta(T_2^*, T_2)^{1\slash 2} & 0 & 0 \\
					\beta(T_2^*, T_2)^{1\slash 2}(T_2^*T_2)^{1\slash 2} & (r^2+r^{-2})-T_2^*T_2 & 0 & 0 \\
					0 & 0 & T_2^*T_2 & \beta(T_2^*, T_2)^{1\slash 2}(T_2^*T_2)^{1\slash 2} \\
					0 & 0 & (T_2^*T_2)^{1\slash 2}\beta(T_2^*, T_2)^{1\slash 2} & (r^2+r^{-2})-T_1^*T_1
				\end{bmatrix}.
			\end{align*}
		}
		Also, $(J_2^*J_2)^{-1}$ equals
		{\small 
		\begin{align*}
			\begin{bmatrix}
				(r^2+r^{-2})-T_2^*T_2 &  -\beta(T_2^*, T_2)^{1\slash 2}(T_2^*T_2)^{1\slash 2} & 0 & 0\\
				-(T_2^*T_2)^{1\slash 2}\beta(T_2^*, T_2)^{1\slash 2} & T_2^*T_2 & 0 & 0\\
				0 & 0 & (r^2+r^{-2})-T_2^*T_2 &  -\beta(T_2^*, T_2)^{1\slash 2}(T_2^*T_2)^{1\slash 2}\\
				0 & 0 & -(T_2^*T_2)^{1\slash 2}\beta(T_2^*, T_2)^{1\slash 2} & T_2^*T_2
			\end{bmatrix}.
		\end{align*}}
		Thus, $J_2^*J_2+(J_2^*J_2)^{-1}=(r^2+r^{-2})I$, and $(J_1, J_2)$ is a doubly commuting pair of quantum annulus unitaries such that 
		$
		J_1^mJ_2^n=T_1^mT_2^n|_{\HS}
		$ 
		for every $m, n \in \Z$. Since the pair $(J_1, J_2)$ admits simultaneously an upper triangular form with respect to $\HS\oplus \HS \oplus \HS \oplus \HS$, we have that
		\[
		J_i^{*n}J_j^m=\begin{bmatrix}
			T_i^{*n} & 0 & 0 & 0 \\
			* & * & * & * \\
			* & * & * & * \\
			* & * & * & * 
		\end{bmatrix}\begin{bmatrix}
			T_j^m & * & * & * \\
			0 & * & * & * \\
			0 & * & * & * \\
			0 & * & * & *
		\end{bmatrix}=\begin{bmatrix}
			T_i^{*n}T_j^m & * & * & * \\
			* & * & * & * \\
			* & * & * & * \\
			* & * & * & *
		\end{bmatrix} \ \ \text{and so,}  \ \ P_{\HS}J_i^{*n}J_j^m|_{\HS}=T_i^{*n}T_j^m
		\]
		for every $m, n \in \Z$ and $1 \leq i, j \leq 2$ with $i \ne j$. The proof is now complete.
	\end{proof}	
	
	The above construction of doubly commuting quantum annulus unitaries for a doubly commuting pair of operators in $Q\A_r$ reveals a pattern that we intend to capitalize to find generalization from $d=2$ case to higher. We now illustrate the same procedure for doubly commuting triples before proceeding to the general $d$-tuple case. Let $\underline{T}=(T_1,T_2,T_3)$ be a doubly commuting triple of operators in $Q\A_r$ acting on a Hilbert space $\HS$. For $ 1\leq i \leq 3$, we denote by
	$
	D_i=T_i(T_i^*T_i)^{-1/2}\beta(T_i^*,T_i)^{1/2}.
	$ 
	Moreover, $J_1,J_2$ and $J_3$ with their inverses are given by
	{\small 
		\setlength{\arraycolsep}{3pt}
		\[
		J_1=
		\begin{bmatrix}
			T_1 & 0 & 0 & 0 & D_1 & 0 & 0 & 0\\
			0 & T_1 & 0 & 0 & 0 & D_1 & 0 & 0\\
			0 & 0 & T_1 & 0 & 0 & 0 & D_1 & 0\\
			0 & 0 & 0 & T_1 & 0 & 0 & 0 & D_1\\
			0 & 0 & 0 & 0 & T_1^{-*} & 0 & 0 & 0\\
			0 & 0 & 0 & 0 & 0 & T_1^{-*} & 0 & 0\\
			0 & 0 & 0 & 0 & 0 & 0 & T_1^{-*} & 0\\
			0 & 0 & 0 & 0 & 0 & 0 & 0 & T_1^{-*}
		\end{bmatrix}, \ \ J_1^{-1}=
		\begin{bmatrix}
			T_1^{-1} & 0 & 0 & 0 & -D_1^* & 0 & 0 & 0\\
			0 & T_1^{-1} & 0 & 0 & 0 & -D_1^* & 0 & 0\\
			0 & 0 & T_1^{-1} & 0 & 0 & 0 & -D_1^* & 0\\
			0 & 0 & 0 & T_1^{-1} & 0 & 0 & 0 & -D_1^*\\
			0 & 0 & 0 & 0 & T_1^{*} & 0 & 0 & 0\\
			0 & 0 & 0 & 0 & 0 & T_1^{*} & 0 & 0\\
			0 & 0 & 0 & 0 & 0 & 0 & T_1^{*} & 0\\
			0 & 0 & 0 & 0 & 0 & 0 & 0 & T_1^{*}
		\end{bmatrix},
		\] 
		\[ 
		J_2=
		\begin{bmatrix}
			T_2 & 0 & D_2 & 0 & 0 & 0 & 0 & 0\\
			0 & T_2 & 0 & D_2 & 0 & 0 & 0 & 0\\
			0 & 0 & T_2^{-*} & 0 & 0 & 0 & 0 & 0\\
			0 & 0 & 0 & T_2^{-*} & 0 & 0 & 0 & 0\\
			0 & 0 & 0 & 0 & T_2 & 0 & D_2 & 0\\
			0 & 0 & 0 & 0 & 0 & T_2 & 0 & D_2\\
			0 & 0 & 0 & 0 & 0 & 0 & T_2^{-*} & 0\\
			0 & 0 & 0 & 0 & 0 & 0 & 0 & T_2^{-*}
		\end{bmatrix}, \ \ J_2^{-1}=
		\begin{bmatrix}
			T_2^{-1} & 0 & -D_2^* & 0 & 0 & 0 & 0 & 0\\
			0 & T_2^{-1} & 0 & -D_2^* & 0 & 0 & 0 & 0\\
			0 & 0 & T_2^{*} & 0 & 0 & 0 & 0 & 0\\
			0 & 0 & 0 & T_2^{*} & 0 & 0 & 0 & 0\\
			0 & 0 & 0 & 0 & T_2^{-1} & 0 & -D_2^* & 0\\
			0 & 0 & 0 & 0 & 0 & T_2^{-1} & 0 & -D_2^*\\
			0 & 0 & 0 & 0 & 0 & 0 & T_2^{*} & 0\\
			0 & 0 & 0 & 0 & 0 & 0 & 0 & T_2^{*}
		\end{bmatrix},
		\]
		and
		\[
		J_3=
		\begin{bmatrix}
			T_3 & D_3 & 0 & 0 & 0 & 0 & 0 & 0\\
			0 & T_3^{-*} & 0 & 0 & 0 & 0 & 0 & 0\\
			0 & 0 & T_3 & D_3 & 0 & 0 & 0 & 0\\
			0 & 0 & 0 & T_3^{-*} & 0 & 0 & 0 & 0\\
			0 & 0 & 0 & 0 & T_3 & D_3 & 0 & 0\\
			0 & 0 & 0 & 0 & 0 & T_3^{-*} & 0 & 0\\
			0 & 0 & 0 & 0 & 0 & 0 & T_3 & D_3\\
			0 & 0 & 0 & 0 & 0 & 0 & 0 & T_3^{-*}
		\end{bmatrix},\ \ J_3^{-1}=
		\begin{bmatrix}
			T_3^{-1} & -D_3^* & 0 & 0 & 0 & 0 & 0 & 0\\
			0 & T_3^{*} & 0 & 0 & 0 & 0 & 0 & 0\\
			0 & 0 & T_3^{-1} & -D_3^* & 0 & 0 & 0 & 0\\
			0 & 0 & 0 & T_3^{*} & 0 & 0 & 0 & 0\\
			0 & 0 & 0 & 0 & T_3^{-1} & -D_3^* & 0 & 0\\
			0 & 0 & 0 & 0 & 0 & T_3^{*} & 0 & 0\\
			0 & 0 & 0 & 0 & 0 & 0 & T_3^{-1} & -D_3^*\\
			0 & 0 & 0 & 0 & 0 & 0 & 0 & T_3^{*}
		\end{bmatrix}.
		\]
	}
	Clearly, $(T_1,T_2,T_3)$ simultaneously extends to $(J_1,J_2,J_3)$. Moreover, $(J_1,J_2,J_3)$ is a doubly commuting triple of quantum annulus unitaries and forms a regular quantum annulus unitary dilation of $(T_1,T_2,T_3)$. A proof of these facts follows from the general $d$-tuple case presented below, which is one of the main results of this article.
	
	\begin{thm}\label{thm_gen}
		Let $\underline{T}=(T_1, \dotsc, T_d)$ be a doubly commuting tuple of operators in $Q\A_r$ acting on a Hilbert space $\HS$. Consider the operators $J_1, \dotsc, J_d$ defined on $\underbrace{\HS \oplus \dotsb \oplus \HS}_{2^d \text{-times}}$ by
		\[
		J_i=I^{\otimes (i-1)} \otimes \begin{bmatrix}
			T_i & T_i(T_i^*T_i)^{-1\slash 2}\beta(T_i^*, T_i)^{1\slash 2}\\
			0 & T_i^{-*}
		\end{bmatrix} \otimes I^{\otimes (d-i)},
		\]
		where $I^{\otimes k}$ denotes the $k$-fold tensor product of the identity operator on $\HS \otimes \C^2 \equiv \HS \oplus \HS$ with itself. Then $\underline{J}=(J_1, \dotsc, J_d)$ is a  doubly commuting tuple of quantum annulus unitaries that simultaneously extends $\underline{T}$, that is, 
		\[
		J_1^{m_1}\dotsc J_d^{m_d}=T_1^{m_1}\dotsc T_d^{m_d}|_\HS 
		\]
		for every $m_1, \dotsc, m_d \in \Z$. Furthermore, $\underline{J}$ is a regular quantum annulus unitary dilation of $\underline{T}$. 
	\end{thm}	
	
	\begin{proof} 	
		Suppose $(T_1, \dotsc, T_d)$ is a doubly commuting tuple of operators in $Q\A_r$ acting on a Hilbert space $\HS$. For $1 \leq i \leq d$, define $D_i=T_i(T_i^*T_i)^{-1\slash 2}\beta(T_i^*, T_i)^{1\slash 2}$. It follows from Theorem \ref{thm_01} that a quantum annulus unitary extension of each $T_i$ is given by
		$
		\Lambda_i=\begin{bmatrix}
			T_i & D_i\\
			0 & T_i^{-*}
		\end{bmatrix}.
		$
		Consider the Hilbert space $\displaystyle \KS=\HS \otimes (\C^2)^{\otimes d}\equiv \HS \oplus \dotsc \oplus \HS$ ($2^d$-times) and define
		\[
		J_i=I^{\otimes (i-1)} \otimes \Lambda_i \otimes I^{\otimes (d-i)} \quad (1 \leq i \leq d).
		\]
		Identify $\HS$ with the subspace $\HS \oplus \{0\} \oplus \dotsc \oplus \{0\} \subseteq \KS$. Since $\Lambda_i$ is a quantum annulus unitary, we have that $\beta(\Lambda_i^*,\Lambda_i)=0$ for $1 \leq i \leq d$. A routine calculation shows that $\beta(J_i^*,J_i)=
		(r^2+r^{-2})-J_i^*J_i-(J_i^*J_i)^{-1}=
		I^{\otimes (i-1)}
		\otimes
		\beta(\Lambda_i^*,\Lambda_i)
		\otimes
		I^{\otimes (d-i)}
		=0$.
		Thus, $(J_1, \dotsc, J_d)$ is a tuple of quantum annulus unitaries. Since $T_1, \dotsc, T_d$ are doubly commuting operators, it follows that $(T_i, T_j^*T_j)$ is a doubly commuting pair for $1 \leq i, j \leq d$ with $i \ne j$. Now, one can apply spectral theorem to show that $T_i$ doubly commute with $(T_j^*T_j)^{1\slash 2}$ and $\beta(T_j^*, T_j)^{1\slash 2}$ with $i \ne j$. Consequently, we have that
		\begin{align}\label{eqn_03}
			T_iD_j=D_jT_i, \quad T_iD_j^*=D_j^*T_i, \quad D_iD_j=D_jD_i \quad \text{and} \quad D_iD_j^*=D_j^*D_i
		\end{align}
		for $1 \leq i, j \leq d$ with $i \ne j$. The operators $J_i$ and $J_j$ act non-trivially only in the $i$-th and $j$-th tensor coordinates, respectively. Hence, to prove that $J_i$ and $J_j$ doubly commute, it suffices to verify that the operator entries of $\Lambda_i$ commute with those of $\Lambda_j$, which directly follows from \eqref{eqn_03} together with the doubly commutativity of $(T_1, \dots, T_d)$. In other words, using \eqref{eqn_03} and arguing as in the proof of Theorem \ref{thm_02}, it follows that $\underline{J}=(J_1,\dotsc,J_d)$ is a doubly commuting tuple. It is easy to see that each $J_i$ admits a block upper triangular matrix representation of size $2^d\times 2^d$ with respect to the decomposition $\underbrace{\HS \oplus \dotsb \oplus \HS}_{2^d \text{-times}}$. In particular, each $J_i$ can be written as
		\[
		J_i=\begin{bmatrix}
			T_i& * & * & \dotsc & * \\
			0 & * & * & \dotsc & *\\
			\vdots & \vdots & \vdots & \dotsc & \vdots\\
			0 & * & * & \dotsc & * 
		\end{bmatrix}.
		\]
		Since the inverse of an invertible block upper triangular operator is again block upper triangular, it follows that $J_i^{m_i}$ is a block upper triangular with $(1,1)$-entry $T_i^{m_i}$ for every $m_i \in \Z$ and $1 \leq i \leq d$. Consequently, we have for every subset $S\subseteq \{1,\dotsc,d\}$ and 
		$m_1,\dotsc,m_d\in\Z$ that
			\begin{align*}
				\left(\prod_{i\in S}J_i^{*m_i}\right)
				\left(\prod_{i\notin S}J_i^{m_i}\right)
				&=
				\begin{bmatrix}
					\left(\underset{i\in S}{\prod}T_i^{*m_i}\right)
					\left(\underset{i\in S}{\prod}T_i^{m_i}\right)
					& * & * & \dotsc & * \\
					* & * & * & \dotsc & *\\
					\vdots & \vdots & \vdots & \dotsc & \vdots\\
					* & * & * & \dotsc & * 
				\end{bmatrix}
			\end{align*}
		and so, 
		\[
		\displaystyle P_{\HS}
		\left(
		\prod_{i\in S}J_i^{*m_i}
		\prod_{i\notin S}J_i^{m_i}
		\right)\Big|_{\HS}
		=
		\left(
		\prod_{i\in S}T_i^{*m_i}
		\prod_{i\notin S}T_i^{m_i}
		\right).
		\]
		Therefore, $\underline{J}$ is a regular quantum annulus unitary dilation of $\underline{T}$.  The proof is now complete.
	\end{proof}

	\section{Annulus as a complete $K$-spectral set for quantum annulus}\label{sec_02}

\noindent Pascoe \cite{Pascoe} proved that $\CA_r$ is a $K_p$-spectral set for operators in $Q\A_r$ with $K_p = 2\left(1+\frac{2r^2}{r^4-1}\right)$, i.e., 
\begin{align}\label{eqn_PB}
	\|g(T)\| \leq 2\left(1+\frac{2r^2}{r^4-1}\right) \|g\|_{\infty, \CA_r}
\end{align}
for every $g \in \text{Rat}(\CA_r)$. An alternative proof of this fact was recently provided in \cite{Pal_Pascoe}. As mentioned by Pascoe in \cite{Pascoe}, Hartz pointed out to him that Shields \cite{Shields} obtained slightly sharper bounds for the norms in the decomposition into two functions on $\D$. This suggests a possible improvement in the spectral constant $K_p$ appearing in \eqref{eqn_PB}. In this section, we present a refinement of the inequality \eqref{eqn_PB} by showing that $\CA_r$ is a complete $K_t$-spectral set for operators in $Q\A_r$, where  
\[
K_t=2\left(1+\frac{2r^2}{(r^2+1)\sqrt{r^4-1}}\right).
\]
Also, we present alternative characterizations for operators in $Q\A_r$ and quantum annulus unitaries. To begin with, we present the following result.

\begin{thm}\label{thm_206}
	$\overline{\mathbb{A}}_r$ is a complete $K_t$-spectral set for every operator in $Q\A_r$, where 
	\[
	K_t=
	2\left(1+\frac{2r^2}{(r^2+1)\sqrt{r^4-1}}\right).
	\]
\end{thm}

\begin{proof}
	Let $T\in Q\mathbb{A}_r$ be acting on a Hilbert space $\HS$. By Theorem \ref{thm_01}, there exists a Hilbert space $\KS\supseteq \HS$ and an operator $J\in Q\mathbb{A}_r$ acting on $\KS$ such that $\beta(J^*,J)=0$	and $g(T)=g(J)|_{\HS}$ for all $g\in \text{Rat}(\overline{\mathbb{A}}_r)$. Fix $n\geq 1$ and let $G=[g_{ij}]_{i, j=1}^n \in M_n(\text{Rat}(\CA_r))$. Then $G(T)=[g_{ij}(T)]_{i, j=1}^n$ acts on $\mathbb{C}^n\otimes \HS$ and
	$G(J)=[g_{ij}(J)]_{i, j=1}^n$ acts on $\mathbb{C}^n\otimes \KS$. For every $1\leq i,j\leq n$, it follows that $G(T)=G(J)\,
	\big|_{\mathbb{C}^n\otimes \HS}$. Consequently, $\|G(T)\| \leq \|G(J)\|$. Now define
	\begin{equation*}
		U=\frac{r}{1+r^2}(J+J^{-*}).
	\end{equation*}
	Since $\beta(J^*,J)=0$, it follows that $U$ is a unitary. Set $\widetilde U=I_n\otimes U$, which is a unitary on $\mathbb{C}^n\otimes \KS$. Therefore,
	$\|G(J)\|=
	\|\widetilde U G(J)\widetilde U^*\|$. Since each $g_{ij}$ is a rational function on $\CA_r$, we can write
	\[
	g_{ij}(z)
	=
	a_{ij}+zg_{ij}^+(z)+z^{-1}g_{ij}^-(z^{-1}),
	\]
	where $a_{ij}+zg_{ij}^+(z)$ is analytic in $|z|<r$ and $
	z^{-1}g_{ij}^-(z^{-1})$ is the principal part. Define
	\[
	A_0=[a_{ij}]_{i, j=1}^n,
	\qquad
	G^+(z)=[g_{ij}^+(z)]_{i, j=1}^n,
	\qquad
	G^-(z)=[g_{ij}^-(z)]_{i, j=1}^n.
	\]
	Then $G(z)=A_0+zG^+(z)+z^{-1}G^-(z^{-1})$. Clearly, $z \mapsto G^+(z)$ and $z \mapsto G^-(z)$ are matrix-valued holomorphic functions on $r\overline{\D}$. Then $G^+(J), G^-(J^{-1}) \in \mathcal{B}(\C^n \otimes \KS)$. For the sake of brevity, we write $(I_n \otimes J)G^+(J)=JG^+(J)$ and $(I_n \otimes J^{-1})G^-(J)=J^{-1}G^-(J^{-1})$, where $I_n$ is the identity matrix in $M_n(\C)$. Thus, $G(J)= A_0+JG^+(J)+J^{-1}G^-(J^{-1})$ and
	\[
		\widetilde U G(J)\widetilde U^*
		=
		\widetilde U
		\left(
		A_0+JG^+(J)+J^{-1}G^-(J^{-1})
		\right)
		\widetilde U^*.
		\]
	Evidently, $
	\widetilde U
	=
	r(1+r^2)^{-1}
	(I_n\otimes J+I_n\otimes J^{-*})$. Using the fact that $G^+(J)$ and $G^-(J^{-1})$ commute with $I_n\otimes J$ and $I_n\otimes J^{-1}$, we have that
	\begin{align*}
		\widetilde U G(J)\widetilde U^*
		=
		\frac{r}{1+r^2}
		\Big[
		\widetilde U(JG^+(J)+A_0)J^*
		+
		J^{-*}(J^{-1}G^-(J^{-1})+A_0)\widetilde U^*
		+
		\widetilde U G^+(J)
		+
		G^-(J^{-1})\widetilde U^*
		\Big].
	\end{align*}
	Since $\|J\|, \|J^{-1}\| \leq r$, it follows that $r\overline{\D}$ is a complete spectral set for $J$ and $J^{-1}$.	Consequently, 
	\begin{align}\label{eqn_2005}
		& \quad \|G(T)\| \notag \\
		& \leq \frac{r}{1+r^2}
		\left[
		\|\widetilde U(JG^+(J)+A_0)J^*\|
		+
		\|J^{-*}(J^{-1}G^-(J^{-1})+A_0)\widetilde U^*\|
		+
		\|\widetilde U G^+(J)\|
		+
		\|G^-(J^{-1})\widetilde U^*\|
		\right] \notag \\
		&	 \leq 
		\frac{r}{1+r^2}
		\left[
		r\|JG^+(J)+A_0\|
		+
		r\|J^{-1}G^-(J^{-1})+A_0\|
		+
		\|G^+(J)\|
		+
		\|G^-(J^{-1})\|
		\right] \notag \\
	& \leq  
\frac{r}{1+r^2}
\left[
r\|zG^+(z)+A_0\|_{\infty, r\overline{\D}}
+
r\|wG^-(w)+A_0\|_{\infty, r\overline{\D}}
+
\|G^+(z)\|_{\infty, r\overline{\D}}
+
\|G^-(w)\|_{\infty, r\overline{\D}}
\right] \notag \\
		& = 
		\frac{r}{1+r^2}
		\left[
		r\|zG^+(z)+A_0\|_{\infty, r\T}
		+
		r\|wG^-(w)+A_0\|_{\infty, r\T}
		+
		\|G^+(z)\|_{\infty, r\T}
		+
		\|G^-(w)\|_{\infty, r\T}
		\right],
	\end{align}	
where the last equality follows from maximum modulus theorem for Banach space valued holomorphic functions on domains in $\C$. We now provide estimates for each norm quantity appearing in \eqref{eqn_2005}. To do so, we make use of the Laurent series representation of $G$. Since $G=[g_{ij}]_{i, j=1}^n$ is a matrix-valued function and each $g_{ij}$ is holomorphic on a neighborhood of $\CA_r$, there exist scalars $s_1, s_2$ satisfying $0<s_1<r^{-1}<r<s_2$ such that $G$ is holomorphic on the annulus $A_{s_1,s_2}=\{z\in\mathbb C:s_1<|z|<s_2\}$. It follows from Theorem 1.9.1 in \cite{Gohberg} that $G$ admits a Laurent expansion
	$\displaystyle 
	G(z)
	=
	\sum_{k=0}^{\infty} B_k z^k
	+
	\sum_{k=1}^{\infty} B_{-k} z^{-k},
	$
	where $B_k \in M_n(\C)$ such that 
	$ \displaystyle \sum_{k=-\infty}^{\infty}\|B_k z^k\|<\infty
	$ for every $z \in \CA_r$. Evidently, $B_0=A_0, G^+(z)=\sum_{k=1}^{\infty} B_k z^{k-1}$ and $G^{-}(z)=\sum_{k=1}^{\infty} B_{-k} z^{k-1}$. Fix $z\in r\T$.  Let $x,y\in \mathbb C^n$ be unit vectors. Define
\[
\varphi_{x,y}: \CA_r \to \C \quad \text{as} \quad \varphi_{x, y}(\zeta)
	=	\langle G(\zeta)x,y\rangle
	=
	\sum_{k=-\infty}^{\infty}
	\langle B_kx,y\rangle \zeta^k.
	\] 
An application of Parseval's identity gives that
	\begin{align}\label{eqn_Parseval}
		\sum_{k=-\infty}^{\infty}
		|\langle B_{k}x,y\rangle|^2 r^{-2k}
		=
		\frac1{2\pi}
		\int_0^{2\pi}
		\left|
		\varphi_{x,y}\left(\frac{e^{i\theta}}{r}\right)
		\right|^2\,d\theta 
		=
		\frac1{2\pi}
		\int_0^{2\pi}
		\left|
		\left\langle
		G\left(\frac{e^{i\theta}}{r}\right)x,y
		\right\rangle
		\right|^2\,d\theta 
		\leq
		\|G\|_{\infty,\CA_r}^2 .
	\end{align}
	Then
{\small 	\begin{align*}
		\left|\left\langle (A_0+zG^+(z))x,y\right\rangle\right|
		=
		\left|
		\left\langle
		G(z)x,y
		\right\rangle
		-
		\sum_{k=1}^{\infty}
		\left\langle B_{-k}x,y\right\rangle z^{-k}
		\right|  
		& \leq
		|\langle G(z)x,y\rangle|
		+
		\sum_{k=1}^{\infty}
		\frac{|\langle B_{-k}x,y\rangle|}{r^k}\\
		& \leq \|G\|_{\infty, \CA_r}+	\left(
		\sum_{k=1}^{\infty}
		|\langle B_{-k}x,y\rangle|^2 r^{2k}
		\right)^{1/2}
		\left(
		\sum_{k=1}^{\infty}
		\frac1{r^{4k}}
		\right)^{1/2} \\
		& \leq \|G\|_{\infty, \CA_r}\left[1+\left(\sum_{k=1}^{\infty}
		\frac1{r^{4k}}
		\right)^{1/2} \right] \quad \text{[by \eqref{eqn_Parseval}]}\\
		&=\|G\|_{\infty,\CA_r}
		\left(
		1+\frac{1}{\sqrt{r^4-1}}
		\right) \qquad \text{and} \\
\left|\left\langle zG^+(z)x,y\right\rangle\right|
	=
	\left|
	\left\langle
	G(z)x,y
	\right\rangle
	-
	\sum_{k=0}^{\infty}
	\left\langle B_{-k}x,y\right\rangle z^{-k}
	\right|  
	 & \leq
	|\langle G(z)x,y\rangle|
	+
	\sum_{k=0}^{\infty}
	\frac{|\langle B_{-k}x,y\rangle|}{r^k}\\
	& \leq \|G\|_{\infty, \CA_r}+	\left(
	\sum_{k=0}^{\infty}
	|\langle B_{-k}x,y\rangle|^2 r^{2k}
	\right)^{1/2}
	\left(
	\sum_{k=0}^{\infty}
	\frac1{r^{4k}}
	\right)^{1/2} \\
	& \leq \|G\|_{\infty, \CA_r}\left[1+\left(\sum_{k=0}^{\infty}
	\frac1{r^{4k}}
	\right)^{1/2} \right] \quad \text{[by \eqref{eqn_Parseval}]}\\
	&=\|G\|_{\infty,\CA_r}
	\left(
	1+\frac{r^2}{\sqrt{r^4-1}}
	\right).
\end{align*}  
}
Consequently, 
		\begin{align*}
	\|A_0+zG^+(z)\|&=\sup_{\|x\|=\|y\|=1}\left|
	\left\langle
	(A_0+zG^+(z))x,y
	\right\rangle
	\right|
	\leq
	\|G\|_{\infty,\CA_r}
	\left(
	1+\frac1{\sqrt{r^4-1}}
	\right) \quad \text{and} \\
	\|zG^+(z)\|&=\sup_{\|x\|=\|y\|=1}\left|
	\left\langle
	(zG^+(z))x,y
	\right\rangle
	\right|
	\leq
	\|G\|_{\infty,\CA_r}
	\left(
	1+\frac{r^2}{\sqrt{r^4-1}}
	\right).
	\end{align*}
Let $w \in r\T$. Arguing similarly as above, one can obtain estimates for $\|A_0+wG^-(w)\|$ and $\|wG^-(w)\|$ so that the following holds:
	\begin{align}\label{eqn_2007}
		\|A_0+zG^+(z)\|_{\infty, r\T}, \ \|A_0+wG^-(w)\|_{\infty, r\T} & \ \leq \|G\|_{\infty, \CA_r}\left(1+\frac{1}{\sqrt{r^4-1}}\right) \quad \text{and} \notag \\
		\|zG^+(z)\|_{\infty, r\T}, \ \|wG^-(w)\|_{\infty, r\T} & \ \leq \|G\|_{\infty, \CA_r}\left(1+\frac{r^2}{\sqrt{r^4-1}}\right).
	\end{align} 
	It now follows from \eqref{eqn_2005} and \eqref{eqn_2007} that
	\begin{align*}
		\|G(T)\|
		& \leq  
		\frac{r}{1+r^2}
		\left[
		r\|zG^+(z)+A_0\|_{\infty, r\T}
		+
		r\|wG^-(w)+A_0\|_{\infty, r\T}
		+
		\|G^+(z)\|_{\infty, r\T}
		+
		\|G^-(w)\|_{\infty, r\T}
		\right]\\
		&\leq
		\frac{r}{1+r^2}
		\bigg[
		2r\left(1+\frac{1}{\sqrt{r^4-1}}\right)
		+
		\frac{2}{r}\left(1+\frac{r^2}{\sqrt{r^4-1}}\right)
		\bigg]
		\|G\|_{\infty,\CA_r}\\
		&= 2\left(1+\frac{2r^2}{(r^2+1)\sqrt{r^4-1}}\right)
		\|G\|_{\infty,\CA_r}.
	\end{align*}
	The proof is now complete.
\end{proof}

Next, we discuss some consequences of Theorem \ref{thm_206} on optimal spectral constant and similarity. For the optimal spectral constant $K(\CA_r)$ for operators in $Q\A_r$, the authors of \cite{Pascoe, Pal_Pascoe} proved that $K(\CA_r) \to 2$ as $r \to \infty$. The next result provides sharper bounds for $K(\CA_r)$, from which its asymptotic behavior follows. 

\begin{cor}
	Suppose $K(\CA_r)$ and $K^{\text{full}}(\CA_r)$ are the smallest constants such that $\CA_r$ is a $K(\CA_r)$-spectral set and a complete $K^{\text{full}}(\CA_r)$-spectral set, respectively, for every operator in $Q\A_r$. Then 
	\[
	2 \leq K(\CA_r) \leq K^{\text{full}}(\CA_r) \leq 2\left(1+\frac{2r^2}{(r^2+1)\sqrt{r^4-1}}\right) \quad \text{and so,} \quad \lim_{r \to \infty}K(\CA_r)=\lim_{r \to \infty}K^{\text{full}}(\CA_r)=2.
	\]
\end{cor}

\begin{proof}
	Tsikalas \cite{TsikalasII} proved that $ 2 \leq K(\CA_r)$. The desired conclusion follows from Theorem \ref{thm_206}.
\end{proof}

Recall that two operators $T_1$ and $T_2$ on a Hilbert space $\HS$ are said to be similar if there exists an invertible operator $S$ on $\HS$ such that $T_2=S^{-1}T_1S$. It follows from the discussion preceding Theorem 11.8 in \cite{Paulsen} that every operator in $Q\A_r$ is similar to an operator having $\CA_r$ as a complete spectral set. Our next result provides a sharper estimate for the similarity constant $\|S\|\cdot \|S^{-1}\|$.

\begin{cor}
	Let $T \in Q\A_r$ be acting on a Hilbert space $\HS$. Then there exists an invertible operator $S$ on $\HS$ such that $\|S\|\cdot \|S^{-1}\| \leq K_t$, where 
	\[
	K_t=2\left(1+\frac{2r^2}{(r^2+1)\sqrt{r^4-1}}\right)
	\] 
	and $S^{-1}TS$ has $\CA_r$ as a complete spectral set. 
\end{cor}

\begin{proof}
	Let $T \in Q\A_r$ be acting on a Hilbert space $\HS$. By Theorem \ref{thm_206}, $\CA_r$ is a complete $K_t$-spectral set. It follows from Corollary 9.12 in \cite{Paulsen} that there exists an invertible operator $S$ such that $\CA_r$ is a
	complete spectral set for $S^{-1}TS$ and $\|S^{-1}\|\cdot \|S\| \leq K_t$, which completes the proof. 
\end{proof}

For an invertible operator $T$, it is proved in \cite{Pas-McCull} that $T \in Q\A_r$ if and only if $\beta(T^*, T)\geq 0$. Moreover, an operator $T$ in $Q\A_r$ dilates to a quantum annulus unitary. By a quantum annulus unitary, we mean an invertible operator $J$ with $\beta(J^*, J)=0$. For a quantum annulus unitary $J$, we considered a unitary in the proof of Theorem \ref{thm_206} given by
\[
U=r(1+r^2)^{-1}(J+J^{-*}).
\] 
Such a unitary corresponding to a quantum annulus unitary $J$ appeared first in \cite{Pal_Pascoe}. In this direction, we define a map $\Phi_r: \C\setminus\{0\} \to \C$ as
\[
\Phi_r(z)=\frac{r}{1+r^2}\left(z+\frac{1}{\overline{z}}\right) \quad \text{and so,} \quad \Phi_r(T)=\frac{r}{1+r^2}(T+T^{-*})
\]
for an invertible operator $T$. The following result characterizes operators in $Q\A_r$ and quantum annulus unitaries in terms of $\Phi_r(T)$.

\begin{thm}\label{thm_204}
	Let $T$ be an invertible operator. Then $T \in Q\A_r$ if and only if $\Phi_r(T)$ is a contraction. Also, $T$ is a quantum annulus unitary if and only if $\Phi_r(T)$ is a unitary. 
\end{thm}

\begin{proof}
	Let $a_r=r(1+r^2)^{-1}$. A straightforward calculation shows that
	\begin{equation}\label{eqn_4001}
		I-\Phi_r(T)^*\Phi_r(T)=a_r^2\left[(r^2+r^{-2})-T^*T-(T^*T)^{-1}\right]=a_r^2\beta(T^*, T).
	\end{equation}
	It was proved in \cite{Pas-McCull} that an invertible operator $T \in Q\A_r$ if and only if $\beta(T^*, T) \geq 0$. It follows from \eqref{eqn_4001} that $T \in Q\A_r$ if and only if $I-\Phi_r(T)^*\Phi_r(T) \geq 0$, that is, $\|\Phi_r(T)\| \leq 1$. Now, suppose $J$ is a quantum annulus unitary. Then $J^*$ is also a quantum annulus unitary since $\beta(J, J^*)=J\beta(J^*, J)J^{-1}$. We have by \eqref{eqn_4001} that
	$I-\Phi_r(J)^*\Phi_r(J)=a_r^2\beta(J^*, J)=0$ and $I-\Phi_r(J)\Phi_r(J)^*=a_r^2\beta(J, J^*)=0$. Thus, $\Phi_r(J)$ is a unitary. The converse follows by reversing the above argument.	
\end{proof}

	\section{Complete $K$-spectral set for doubly commuting operators in $Q\A_r$}\label{sec_04}
	
	\noindent The authors of \cite{Pal_Pascoe} proved that the polyannulus $\CA_r^d$ is a $K_{dc}$-spectral set for doubly commuting $d$-tuples of operators in $Q\A_r$, where 
	\[
	K_{dc}=\left(\frac{3r^2-1}{r^2-1}\right)^d.
	\] 
	Also, it was proved as Theorem 4.3 in \cite{Pal_Pascoe} that if $K_{dc}(\CA_r^d)$ denote the smallest constant for which $\CA_r^d$ is a $K_{dc}(\CA_r^d)$-spectral set for every doubly commuting $d$-tuple in $Q\A_r$, then 
	\[
	2^d \leq K_{dc}(\CA_r^d) \leq K_{dc} \quad \text{and so,} \quad 2^d \leq \underset{r \to \infty}{\lim}K_{dc}(\CA_r^d) \leq 3^d.
	\] 
In the case $d=1$, that is, for $T \in Q\A_r$, we have that $\|g(T)\| \leq \left[(3r^2-1)/(r^2-1)\right]\|g\|_{\infty, \CA_r}$ for every $g \in \text{Rat}(\CA_r)$. Comparing the constant $(3r^2-1)/(r^2-1)$ with the constant
\[
K_t= 2\left(1+\frac{2r^2}{(r^2+1)\sqrt{r^4-1}}\right)
\]
obtained in Theorem \ref{thm_206}, we have that
\[
2\left(1+\frac{2r^2}{(r^2+1)\sqrt{r^4-1}}\right) \lneq \frac{3r^2-1}{r^2-1}.
\]
Moreover, as $r \to \infty$, the left-hand side tends to $2$, while the right-hand side tends to $3$. This indicates that the bound $K_{dc}=\left[(3r^2-1)/(r^2-1)\right]^d$ is not optimal in the multivariable setting. This gap arises from the fact that the proof of Theorem 4.3 in \cite{Pal_Pascoe} is independent of dilation theorem for doubly commuting operators in $Q\mathbb{A}_r$, unlike the single-operator case presented in Theorem \ref{thm_206}. This naturally leads to the following questions:

\begin{enumerate}[leftmargin=*] 
	
	\item Is the closed polyannulus $\CA_r^d$ is a complete $K$-spectral set for doubly commuting $d$-tuples of operators in $Q\A_r$? 
	
	\item Can the bound $K_{dc}$ be improved to a smaller constant $K$ so that $\CA_r^d$ becomes a $K$-spectral set (or, a complete $K$-spectral set) for doubly commuting $d$-tuples of operators in $Q\A_r$?

\item Let $K_{dc}(\CA_r^d)$ and $K_{dc}^{\mathrm{full}}(\CA_r^d)$ denote the smallest constants for which $\CA_r^d$ is a $K_{dc}(\CA_r^d)$-spectral set and a complete $K_{dc}^{\mathrm{full}}(\CA_r^d)$-spectral set, respectively, for every doubly commuting $d$-tuple of operators in $Q\A_r$. Then what is the asymptotic behavior of these optimal constants?
\end{enumerate}

In this section, we provide answers to these questions. We first establish that
	\begin{align}\label{eqn_pretext}
	\|G(T_1, \dotsc, T_d)\| \leq \left[2\left(1+\frac{2r^2}{(r^2+1)\sqrt{r^4-1}}\right)\right]^d\|G\|_{\infty, \CA_r^d}
	\end{align}
	for every doubly commuting $d$-tuple of operators in $Q\A_r$ and $G \in M_n(\text{Rat}(\CA_r^d))$. The proof capitalizes on the dilation theorem for doubly commuting operators in $Q\A_r$ presented in Theorems \ref{thm_mainQAr} and \ref{thm_gen}. As a consequence, we obtain sharper bounds for the optimal complete spectral constant  $K_{dc}^{\text{full}}(\CA_r^d)$ showing that
	$
	\underset{r \to \infty}{\lim}K_{dc}^{\text{full}}(\CA_r^d)=2^d.
	$
We also provide analogous estimates for the optimal $K$-spectral constant (not necessarily complete) for doubly commuting tuples in $Q\A_r$. For the convenience of the reader, we first establish \eqref{eqn_pretext} in the case $n=1$, that is, when $G \in \text{Rat}(\CA_r^d)$.
	
	\begin{thm}\label{thm_401}
		Let $(T_1, \dotsc, T_d)$ be a doubly commuting tuple of operators in $Q\A_r$ acting on a Hilbert space $\HS$ and $g \in \text{Rat}(\CA_r^d)$. Then
		\[
		\|g(T_1, \dotsc, T_d)\| \leq \left[2\left(1+\frac{2r^2}{(r^2+1)\sqrt{r^4-1}}\right)\right]^d\|g\|_{\infty, \CA_r^d}.
		\]
	\end{thm}
	
	\begin{proof}
		We have by Theorem \ref{thm_gen} that there exists a doubly commuting tuple $(J_1, \dotsc, J_d)$ of quantum annulus unitaries acting on a Hilbert space $\KS \supseteq \HS$ such that 
		\[
		g(T_1, \dotsc, T_d)=g(J_1, \dotsc, J_d)|_{\HS}
		\] 
		for every $g \in \text{Rat}(\CA_r^d)$. Let us define
		\[
		(U_1, \dotsc, U_d)=\left(\frac{r}{1+r^2}(J_1+J_1^{-*}), \ \dotsc, \ \frac{r}{1+r^2}(J_d+J_d^{-*}) \right). 
		\]
		It follows from Theorem \ref{thm_204} that each $U_i$ is a unitary on $\KS$. Moreover, $(U_1, \dotsc, U_d)$ is a doubly commuting tuple, which follows from the doubly commutativity of the tuple $(J_1, \dotsc, J_d)$. To clarify the argument, we restrict attention to the case $d=2$ and divide the proof into several steps. The general case follows by an inductive argument.
		
		\medskip 
		
		\noindent \textit{Step 1.} Let $g \in \text{Rat}(\CA_r^2)$. Since $g$ is holomorphic on a neighborhood of $\CA_r^2$, there exist scalars $s_1,s_2$ satisfying $
		0<s_1<r^{-1}<r<s_2$ such that $g$ is holomorphic on $A_{s_1,s_2}^2=A_{s_1,s_2} \times A_{s_1,s_2}$, where
		$A_{s_1,s_2}=\{z\in \mathbb{C}: s_1<|z|<s_2\}$. Thus, $g$ admits the Laurent series representation given by
		\begin{equation}\label{eqn_0301}
			g(z_1, z_2)=\overset{\infty}{\underset{n=-\infty}{\sum}}\ \overset{\infty}{\underset{m=-\infty}{\sum}}a_{n,m} \ z_1^nz_2^m \qquad \text{for} \  (z_1, z_2) \in A_{s_1, s_2}^2.
		\end{equation}
		The above series converges absolutely and uniformly on $\CA_r^2$. Therefore, we can re-write
		\[
		g(z_1, z_2)=\overset{\infty}{\underset{n=-\infty}{\sum}}b_n(z_2)z_1^n,
		\quad \text{where} \quad
		b_n(z_2)=\overset{\infty}{\underset{m=-\infty}{\sum}}a_{n,m}z_2^m.
		\]
		Moreover, each $b_n$ is holomorphic on $\CA_r$ and so, $b_n(J_2)$ is a well-defined operator acting on $\KS$. One can refer to Chapter II in \cite{Range} for further details. Define a holomorphic map $F: \CA_r \to \mathcal{B}(\KS)$ as $F(z_1)=g(z_1, J_2)=\overset{\infty}{\underset{n=-\infty}{\sum}}b_n(J_2)z_1^n$. Using the same decomposition argument as in Theorem \ref{thm_206} to the $\mathcal{B}(\KS)$-valued holomorphic map $F$, we can write
		\[
		F(z_1)=\overset{\infty}{\underset{k=-\infty}{\sum}}B_kz_1^k=\overset{\infty}{\underset{k=0}{\sum}}B_kz_1^k+\overset{\infty}{\underset{k=1}{\sum}}B_{-k}z_1^{-k}=B_0+z_1F^+(z_1)+z_1^{-1}F^-(z_1^{-1}),
		\]
		where $\displaystyle B_k=b_k(J_2), \, F^+(z_1)=\overset{\infty}{\underset{n=1}{\sum}}b_n(J_2)z_1^{n-1}$ and $\displaystyle F^-(z_1)=\overset{\infty}{\underset{n=1}{\sum}}b_{-n}(J_2)z_1^{n-1}$. Let $z_1 \in \CA_r$. Then
		\[
		\overset{\infty}{\underset{k=-\infty}{\sum}}\left\|B_kz_1^k\right\| =\overset{\infty}{\underset{k=-\infty}{\sum}}\left\|b_k(J_2)z_1^k\right\|=\overset{\infty}{\underset{k=-\infty}{\sum}}\left\|\overset{\infty}{\underset{m=-\infty}{\sum}}a_{k,m}J_2^mz_1^k\right\|\leq \overset{\infty}{\underset{k, m=-\infty}{\sum}}r^{|m|}|a_{k, m}|\cdot|z_1|^k <\infty 
		\]
as the Laurent series in \eqref{eqn_0301} converges absolutely on $\CA_r^2$ and $r^{-1} \leq \|J_2\| \leq r$. Next, we claim that
		\begin{align}\label{eqn_4002}
			U_1F(J_1)U_1^*=a_r\left[U_1\left(B_0+J_1F^+(J_1)\right)J_1^*+J_1^{-*}\left(B_0+J_1^{-1}F^-(J_1^{-1})\right)U_1^*+U_1F^+(J_1)+F^-(J_1^{-1})U_1^*\right], \notag \\
		\end{align}
		where $a_r=r(1+r^2)^{-1}$. Also, $F(J_1)=\overset{\infty}{\underset{n=-\infty}{\sum}}b_n(J_2)J_1^n, \  F^+(J_1)=\overset{\infty}{\underset{n=1}{\sum}}b_n(J_2)J_1^{n-1}$ and $ F^-(J_1^{-1})=\overset{\infty}{\underset{n=1}{\sum}}b_{-n}(J_2)J_1^{-n+1}$, which are operators on $\KS$. A routine computation shows that
		\begin{align}\label{eqn_4003}
			& U_1J_1F^+(J_1)U_1^*=a_rU_1J_1F^+(J_1)(J_1^*+J_1^{-1})=a_r\left[U_1J_1F^+(J_1)J_1^*+U_1F^+(J_1)\right], \notag \\
			&
			U_1J_1^{-1}F^-(J_1^{-1})U_1^*=a_r(J_1+J_1^{-*})J_1^{-1}F^-(J_1^{-1})U_1^*=a_r\left[F^-(J_1^{-1})U_1^*+J_1^{-*}J_1^{-1}F^-(J_1^{-1})U_1^*\right] \ \text{and} \notag \\
			&
			U_1U_1^*=a_r^2(J_1+J_1^{-*})(J_1^*+J_1^{-1})=a_r^2(J_1J_1^*+2J_1^{-*}J_1^*+J_1^{-*}J_1^{-1})=a_r\left(U_1J_1^*+J_1^{-*}U_1^*\right).
		\end{align}
		As $J_2$ doubly commutes with $J_1$ and $U_1$, the same holds for $B_0=b_0(J_2)$. Then
		\begin{align*}
			& \quad \ U_1F(J_1)U_1^*\\
			&=U_1\left[B_0+J_1F^+(J_1)+J_1^{-1}F^-(J_1^{-1})\right]U_1^*\\
			&=U_1B_0U_1^*+U_1J_1F^+(J_1)U_1^*+U_1J_1^{-1}F^-(J_1^{-1})U_1^*\\
			&=U_1U_1^*B_0+U_1J_1F^+(J_1)U_1^*+U_1J_1^{-1}F^-(J_1^{-1})U_1^* \quad \quad  &\tag*{[\text{since $B_0, U_1$ doubly commute}]}\\
			&=a_r\left(U_1J_1^*B_0+J_1^{-*}U_1^*B_0+U_1J_1F^+(J_1)J_1^*+U_1F^+(J_1)+F^-(J_1^{-1})U_1^*+J_1^{-*}J_1^{-1}F^-(J_1^{-1})U_1^*\right) \\
			& \tag*{[by \eqref{eqn_4003}]}\\
			&=a_r\left(U_1B_0J_1^*+J_1^{-*}B_0U_1^*+U_1J_1F^+(J_1)J_1^*+U_1F^+(J_1)+F^-(J_1^{-1})U_1^*+J_1^{-*}J_1^{-1}F^-(J_1^{-1})U_1^*\right)\\
			& \tag*{[\text{since $B_0$ doubly commutes with $J_1, U_1$}]}\\
			&=a_r\bigg[U_1\left(B_0+J_1F^+(J_1)\right)J_1^*+J_1^{-*}\left(B_0+J_1^{-1}F^-(J_1^{-1})\right)U_1^*+U_1F^+(J_1)+F^-(J_1^{-1})U_1^*\bigg]
		\end{align*}
		and so the claim in \eqref{eqn_4002} holds. 
		
		\medskip 
		
		\noindent \textit{Step 2.} We have by Proposition 2.3 in \cite{Hartz} that if $p(z)=A_0+\dotsc + A_nz^n$ is a polynomial with operator coefficients, $T$ is a contraction that doubly commutes with each $A_k$ and $p(T)=A_0+A_1T+\dotsc+A_nT^n$, then $\|p(T)\| \leq \|p\|_{\infty, \overline{\D}}=\|p\|_{\infty, \T}$, where the last equality follows from maximum principle for vector-valued holomorphic functions (see Theorem 1.2.1 in \cite{Gohberg}). Using approximation arguments, one can show that if $f(z)=\sum_{k=0}^\infty A_kz^k$ such that the series converges absolutely on $\overline{\D}$ and each $A_k$ doubly commutes with a given contraction $T$, then 
		\[
		\|f(T)\|=\left\|\overset{\infty}{\underset{k=0}{\sum}}A_kT^k\right\| \leq \|f\|_{\infty, \T}.
		\]
		Note that
		\begin{small}\begin{align*}
				B_0+zF^+(z)=\overset{\infty}{\underset{k=0}{\sum}}B_kz^k, \ B_0+wF^-(w)=\overset{\infty}{\underset{k=0}{\sum}}B_{-k}w^{k}, \ F^+(z)=\overset{\infty}{\underset{k=1}{\sum}}B_kz^{k-1} \ \text{and} \ F^-(w)=\overset{\infty}{\underset{k=1}{\sum}}B_{-k}w^{k-1}. 
			\end{align*}
		\end{small}
		Each of the series above converges absolutely on $r\overline{\D}$ and the coefficients $B_k$ doubly commute with $J_1$. Since $\|J_1\|, \|J_1^{-1}\| \leq r$, a re-scaling argument together with Proposition 2.3 in \cite{Hartz} gives
		\begin{align}\label{eqn_4004}
			& \| B_0+J_1F^+(J_1)\| \leq \| B_0+zF^+(z)\|_{\infty, r\T}, \quad \|B_0+J_1^{-1}F^-(J_1^{-1})\| \leq \|B_0+wF^-(w)\|_{\infty, r\T}, \notag \\
			& \|F^+(J_1)\|\leq  \|F^+(z)\|_{\infty, r\T} \quad \text{and} \quad \|F^-(J_1^{-1})\| \leq \|F^-(w)\|_{\infty, r\T}.
		\end{align}
In comparison with the proof of Theorem \ref{thm_206}, we note that the coefficients $B_k$ therein are matrices, whereas in the present setting each $B_k$ is as an operator on $\KS$. Consequently, by following arguments analogous to those used in the proof of Theorem \ref{thm_206}, we obtain the following estimates:
			\begin{align}\label{eqn_4005}
			\|B_0+zF^+(z)\|_{\infty, r\T}, \ \|B_0+wF^-(w)\|_{\infty, r\T} & \ \leq \|F\|_{\infty, \CA_r}\left(1+\frac{1}{\sqrt{r^4-1}}\right) \quad \text{and} \notag \\
			\|zF^+(z)\|_{\infty, r\T}, \ \|wF^-(w)\|_{\infty, r\T} & \ \leq \|F\|_{\infty, \CA_r}\left(1+\frac{r^2}{\sqrt{r^4-1}}\right).
		\end{align}

		\medskip 
		
		\noindent \textit{Step 3.} We are now in a position to prove the final estimate. To do so, note that
		\begin{align*}
			& \ \quad \|F(J_1)\|\\
			&=\|U_1F(J_1)U_1^*\|\\
			&=a_r\left\|\left[U_1\left(B_0+J_1F^+(J_1)\right)J_1^*+J_1^{-*}\left(B_0+J_1^{-1}F^-(J_1^{-1})\right)U_1^*+U_1F^+(J_1)+F^-(J_1^{-1})U_1^*\right]\right\| \tag*{\text{[by \eqref{eqn_4002}]}}\\
			& \leq a_r \left[ r\| B_0+J_1F^+(J_1)\|+r\|B_0+J_1^{-1}F^-(J_1^{-1})\|+\|F^+(J_1)\|+\|F^-(J_1^{-1})\| \right] \smallskip \tag*{[\text{since $\|J_1\| \leq r$}]}\\
			& \leq a_r \left[ r\| B_0+zF^+(z)\|_{\infty, r\T}+r\|B_0+wF^-(w)\|_{\infty, r\T}+\|F^+(z)\|_{\infty, r\T}+\|F^-(w)\|_{\infty, r\T} \right] \tag*{[\text{by \eqref{eqn_4004}}]}\\
		&\leq
		\frac{r}{1+r^2}
		\bigg[
		2r\left(1+\frac{1}{\sqrt{r^4-1}}\right)
		+
		\frac{2}{r}\left(1+\frac{r^2}{\sqrt{r^4-1}}\right)
		\bigg]
		\|F\|_{\infty,\CA_r}\\
		&= 2\left(1+\frac{2r^2}{(r^2+1)\sqrt{r^4-1}}\right)
		\|F\|_{\infty,\CA_r}.
		\end{align*}
		Fix $z_1 \in \CA_r$ and define a holomorphic map $g_{z_1}: \CA_r \to \C$ as $g_{z_1}(z_2)=g(z_1, z_2)$. Since $g_{z_1}$ is a complex-valued holomorphic map on $\CA_r$ and $J_2 \in Q\A_r$, we have by Theorem \ref{thm_206} that
		\begin{align*}
			\|F(z_1)\|=\|g(z_1, J_2)\|
			=\|g_{z_1}(J_2)\|
			& \leq 2\left(1+\frac{2r^2}{(r^2+1)\sqrt{r^4-1}}\right)
			\|g_{z_1}\|_{\infty,\CA_r} \\
			& \leq 2\left(1+\frac{2r^2}{(r^2+1)\sqrt{r^4-1}}\right)
			\|g\|_{\infty,\CA_r^2}
		\end{align*}
		and so, $\|F\|_{\infty, \CA_r} \leq 2\left(1+\frac{2r^2}{(r^2+1)\sqrt{r^4-1}}\right)
		\|g\|_{\infty,\CA_r^2}$. Finally, we have that
		\[
		\|g(J_1, J_2)\| =\|F(J_1)\| \leq 2\left(1+\frac{2r^2}{(r^2+1)\sqrt{r^4-1}}\right)
		\|F\|_{\infty,\CA_r} \leq \left[2\left(1+\frac{2r^2}{(r^2+1)\sqrt{r^4-1}}\right)\right]^2
		\|g\|_{\infty,\CA_r^2}.
		\]
		The proof is now complete.
	\end{proof}
	
If we denote the smallest spectral constant by $K_{dc}(\overline{\mathbb{A}}_r^d)$ for doubly commuting $d$-tuples in $Q\mathbb{A}_r$, then the resulting bounds from \cite{Pal_Pascoe} are given by
\[
2^d \leq K_{dc}(\overline{\mathbb{A}}_r^d) \leq \left(\frac{3r^2-1}{r^2-1}\right)^d
\]
and so, $2^d \leq \underset{r \to \infty}{\lim} K_{dc}(\overline{\mathbb{A}}_r^d) \leq 3^d$. An application of Theorem \ref{thm_401} sharpens these bounds as follows. 

	\begin{cor}\label{cor_402}
		Let $K_{dc}(\CA_r^d)$ be the smallest constant for which the polyannulus $\CA_r^d$ is a $K_{dc}(\CA_r^d)$-spectral set for every doubly commuting operator $d$-tuple in $Q\A_r$. Then 
		\[
		2^d \leq K_{dc}(\CA_r^d) \leq \left[2\left(1+\frac{2r^2}{(r^2+1)\sqrt{r^4-1}}\right)\right]^d \quad \text{and} \quad \displaystyle \lim_{r \to \infty} K_{dc}(\CA_r^d)= 2^d.
		\] 
	\end{cor}

\begin{proof}
By Theorem 5.1 in \cite{Pal_Pascoe}, $2^d \leq K_{dc}(\CA_r^d)$. The conclusion now follows from Theorem \ref{thm_401}.
\end{proof}
	We now show that the estimate obtained in Theorem \ref{thm_401} holds for every matrix-valued rational function on $\CA_r^d$, thereby establishing a complete $K$-spectral set analog of Theorem \ref{thm_401}.
	
	\begin{thm}\label{thm_403}
		Let $(T_1,\ldots,T_d)$ be a doubly commuting tuple of operators in $Q\A_r$ acting on a Hilbert space $\HS$. Then $\CA_r^d$ is a complete $K_{dc}^{(d)}$-spectral set for $(T_1,\ldots,T_d)$, where
		\[
		K_{dc}^{(d)}
		=
		\left[2\left(1+\frac{2r^2}{(r^2+1)\sqrt{r^4-1}}\right)
		\right]^d.
		\]
	\end{thm}
	
\begin{proof}
	We proceed by induction on $d$. The case $d=1$ follows from Theorem \ref{thm_206}. Let $d\ge 2$ and assume that $\CA_r^{d-1}$ is a complete $K_{dc}^{(d-1)}$-spectral set for every doubly commuting $(d-1)$-tuple of operators in $Q\A_r$. Let $(T_1,\dotsc, T_d)$ be a doubly commuting tuple of operators in $Q\A_r$ and let
	$
	G=[g_{ij}] \in M_n\big(\text{Rat}(\CA_r^d)\big).
	$
	By Theorem \ref{thm_gen}, there exists a doubly commuting tuple
		$
		(J_1,\ldots,J_d)
		$
		of quantum annulus unitaries on a Hilbert space $\KS\supseteq\HS$ such that
		$
		g(T_1,\ldots,T_d)
		=
		g(J_1,\ldots,J_d)\big|_{\HS}
		$
		for every $g\in\text{Rat}(\CA_r^d)$. Consequently, $G(T_1,\ldots,T_d)
		=
		G(J_1,\ldots,J_d)\big|_{\C^n \otimes \HS}$ and so,
		$
		\|G(T_1,\ldots,T_d)\|
		\leq
		\|G(J_1,\ldots,J_d)\|.
		$
		A repeated application of one variable Laurent series representation for $M_n(\C)$-valued holomorphic functions on an annulus gives that	
		\[
		G(z_1,\dots,z_d)=\sum_{\nu_1, \dotsc, \nu_d \in\mathbb{Z}} A_{\nu_1, \dotsc, \nu_d} z_1^{\nu_1}\cdots z_d^{\nu_d} \ \text{and so,} \ G(J_1, \dotsc, J_d)=\sum_{\nu_1, \dotsc, \nu_d \in \Z}A_{\nu_1, \dotsc, \nu_d} \otimes (J_1^{\nu_1}\dotsc J_d^{\nu_d})
		\]
is an operator on $\C^n \otimes \KS$. The coefficients are explicitly given by
\[
 A_{\nu_1, \dotsc, \nu_d}=
		\frac{1}{(2\pi i)^d}
		\int_{|\zeta_1|=\rho_1}\cdots\int_{|\zeta_d|=\rho_d}
		\frac{G(\zeta_1,\dots,\zeta_d)}
		{\zeta_1^{\nu_1+1}\cdots \zeta_d^{\nu_d+1}}
		\,d\zeta_1\cdots d\zeta_d,
\]
where $ r^{-1}<\rho_j<r$ and $1 \leq j \leq d$. Furthermore, the Laurent series of $G$ converges uniformly and absolutely on $\CA_r^d$. Define $F: \CA_r \to \mathcal{B}(\C^n \otimes \KS)$ as $F(z_1)=G(z_1, J_2, \dotsc, J_d)$, which is a holomorphic map. The map $F$ can be re-written as
		\[
		F(z_1)=\sum_{\nu_1\in\mathbb{Z}} B_{\nu_1}(J_2, \dotsc, J_d)z_1^{\nu_1}, \quad \text{where} \quad B_{\nu_1}(J_2, \dotsc, J_d)=\sum_{\nu_2, \dotsc, \nu_d \in \Z}A_{\nu_1, \nu_2, \dotsc, \nu_d} \otimes (J_2^{\nu_2}\dotsc J_d^{\nu_d}).
		\]
The decomposition of $F$ into analytic and principal parts is given by
	\[
	F(z_1)=B_0(J_2, \dotsc, J_d)+z_1F^+(z_1)+z_1^{-1}F^-(1\slash z_1)=B_0(J_2, \dotsc, J_d)+F_1(z_1)+F_2(1\slash z_1)
	\]	
where  $F_1(z_1)=z_1F^+(z_1)$ and $F_2(z_1)=z_1F^-(z_1)$ with
\[
F^+(z_1)=\overset{\infty}{\underset{\nu_1=1}{\sum}} B_{\nu_1}(J_2, \dotsc, J_d)z_1^{\nu_1-1} \quad \text{and} \quad F^-(z_1)=\overset{-1}{\underset{\nu_1=-\infty}{\sum}} B_{-\nu_1}(J_2, \dotsc, J_d)z_1^{\nu_1-1}.
\] 
Let $\widetilde{J}_1=I_n \otimes J_1$. Then
\begin{align*}
F_1(J_1)
=\widetilde J_1F^+(J_1)
&=(I_n \otimes J_1)\overset{\infty}{\underset{\nu_1=1}{\sum}} B_{\nu_1}(J_2, \dotsc, J_d)(I_n \otimes J_1^{\nu_1-1})\\
&=\overset{\infty}{\underset{\nu_1=1}{\sum}} \sum_{\nu_2, \dotsc, \nu_d \in \Z}A_{\nu_1, \nu_2, \dotsc, \nu_d} \otimes (J_1^{\nu_1} J_2^{\nu_2}\dotsc J_d^{\nu_d})
\end{align*}
and
\begin{align*} 
F_2(J_1^{-1})
=\widetilde J_1^{-1}F^-(J_1^{-1})
&=(I_n \otimes J_1^{-1})\overset{-1}{\underset{\nu_1=-\infty}{\sum}} B_{-\nu_1}(J_2, \dotsc, J_d)(I_n \otimes J_1^{\nu_1-1})^{-1}\\
&=\overset{-1}{\underset{\nu_1=-\infty}{\sum}} \sum_{\nu_2, \dotsc, \nu_d \in \Z}A_{-\nu_1, \nu_2, \dotsc, \nu_d} \otimes \left(J_1^{-\nu_1} J_2^{\nu_2}\dotsc J_d^{\nu_d}\right).
\end{align*}
Let $z_1 \in \CA_r$. Let $a_r=r(1+r^2)^{-1}$ and $U_1=a_r(J_1+J_1^{-*})$. Since $J_1, \dotsc, J_d$ are doubly commuting operators, it follows that $B_{\nu_1}(J_2, \dotsc, J_d)$ doubly commutes with $\widetilde J_1=I_n\otimes J_1$ and $\widetilde U_1=I_n \otimes U_1$. Following the proof of \eqref{eqn_4002} in Theorem \ref{thm_401}, one can show that
\begin{align*}
	& \quad \widetilde U_1F(J_1)\widetilde U_1^*\\
	&=a_r\left[\widetilde U_1\left(B_0+\widetilde J_1F^+(J_1)\right)(\widetilde J_1)^*+(\widetilde J_1)^{-*}\left(B_0+(\widetilde J_1)^{-1}F^-(J_1^{-1})\right)\widetilde U_1^*+\widetilde U_1F^+(J_1)+F^-(J_1^{-1})\widetilde U_1^*\right]
\end{align*}
and thus, 
\begin{align*}
 \|F(J_1)\| 	
	& \leq a_r\left(r\|B_0+\widetilde J_1F^+(J_1)\|+r\|B_0+(\widetilde J_1)^{-1}F^-(J_1^{-1})\|+\|F^+(J_1)\|+\|F^-(J_1^{-1})\|\right)\\
	& = a_r\left(r\|B_0+F_1(J_1)\|+r\|B_0+F_2(J_1^{-1})\|+\|F^+(J_1)\|+\|F^-(J_1^{-1})\|\right).
\end{align*}
We now wish to apply von Neumann's type inequality to each of the norm entity appearing in the last equality. For the first entity, we consider the map $q: r\overline{\D} \to \mathcal{B}(\C^n \otimes \KS)$ defined as 
$
q(z_1)= B_0+F_1(z_1)=B_0+z_1F^+(z_1)=\overset{\infty}{\underset{\nu_1=0}{\sum}} B_{\nu_1}(J_2, \dotsc, J_d)z_1^{\nu_1}.
$
Then $q$ is a holomorphic function on $r\overline{\D}$. Also, the doubly commutativity of $J_1, \dotsc, J_d$ implies that $I_n \otimes J_1$ doubly commutes with each $B_{\nu_1}(J_2, \dotsc, J_d)$. Using the same application of Proposition 2.3 from \cite{Hartz} as employed in the proof of Theorem \ref{thm_401}, we have that
\[
\|B_0+F_1(J_1)\|=\|q(I_n \otimes J_1)\| \leq \|q\|_{\infty, r\overline{\D}}=\|q\|_{\infty, r\T}=\|B_0+z_1F^+(z_1)\|_{\infty, r\T}.
\]
For a pair of unit vectors $x, y \in \C^n \otimes \KS$, define 
\[
\varphi_{x, y}: \CA_r \to \C \quad \text{as} \quad \varphi_{x, y}(\zeta)=\left\langle F(\zeta)x, y \right \rangle	=
\sum_{\nu_1=-\infty}^{\infty}
\langle B_{v_1}(J_2, \dotsc, J_d)x,y\rangle \zeta^{\nu_1}.
\] 
By Parseval's identity, we have that
\begin{align}\label{eqn_Parseval1}
	\sum_{\nu_1=-\infty}^{\infty}
	|\langle B_{\nu_1}(J_2, \dotsc, J_d)x,y\rangle|^2 r^{-2\nu_1}
	=
	\frac1{2\pi}
	\int_0^{2\pi}
	\left|
	\varphi_{x,y}\left(\frac{e^{i\theta}}{r}\right)
	\right|^2\,d\theta 
	\leq
	\|F\|_{\infty,\CA_r}^2. 
\end{align}
For every $z_1 \in r\T$, it follows that
\begin{align*}
		\left|\left\langle (B_0+z_1F^+(z_1))x,y\right\rangle\right|
		&=
		\left|
		\left\langle
		F(z_1)x,y
		\right\rangle
		-
		\sum_{\nu_1=1}^{\infty}
		\left\langle B_{-\nu_1}(J_2, \dotsc, J_d)x,y\right\rangle z^{-\nu_1}
		\right|  \\
		& \leq
		|\langle F(z)x,y\rangle|
		+
		\sum_{\nu_1=1}^{\infty}
		\frac{|\langle B_{-\nu_1}(J_2, \dotsc, J_d)x,y\rangle|}{r^\nu_1}\\
		& \leq \|F\|_{\infty, \CA_r}+	\left(
		\sum_{\nu_1=1}^{\infty}
		|\langle B_{-\nu_1}(J_2, \dotsc, J_d)x,y\rangle|^2 r^{2\nu_1}
		\right)^{1/2}
		\left(
		\sum_{\nu_1=1}^{\infty}
		\frac1{r^{4\nu_1}}
		\right)^{1/2} \\
		& \leq \|F\|_{\infty, \CA_r}\left[1+\left(\sum_{\nu_1=1}^{\infty}
		\frac1{r^{4\nu_1}}
		\right)^{1/2} \right] \quad \text{[by \eqref{eqn_Parseval1}]}\\
		&=\|F\|_{\infty,\CA_r}
		\left(
		1+\frac{1}{\sqrt{r^4-1}}
		\right)
	\end{align*} 
and so, 
\begin{align*}
\|B_0+F_1(J_1)\| \leq \|B_0+z_1F^+(z_1)\|_{\infty, r\T}
&=\underset{z_1 \in r\T}{\sup}\left[\underset{\|x\|=\|y\|=1}{\sup}\left|\left\langle (B_0+z_1F^+(z_1))x,y\right\rangle\right|\right] \\
& \leq \|F\|_{\infty,\CA_r}
\left(
1+\frac{1}{\sqrt{r^4-1}}
\right) .
\end{align*}
Similarly, one can show that 
{\small 
	\[	
\|B_0+F_2(J_1^{-1})\| \leq \|F\|_{\infty, \CA_r}\left(1+\frac{1}{\sqrt{r^4-1}}\right)
 \ \text{and} \ 
\|F^+(J_1)\|, \ \|F^-(J_1^{-1})\|_{\infty, r\T}   \leq \frac{\|F\|_{\infty, \CA_r}}{r}\left(1+\frac{r^2}{\sqrt{r^4-1}}\right).
\]
}
Combining everything together, we have
\begin{align*} 
		\|G(T_1,\dots, T_d)\|
		& \leq \|G(J_1,\dots,J_d)\|\\
		&=\|F(J_1)\|\\
		& \leq  a_r\left(r\|B_0+F_1(J_1)\|+r\|B_0+F_2(J_1^{-1})\|+\|F^+(J_1)\|+\|F^-(J_1^{-1})\|\right)\\
	&\leq
	\frac{r}{1+r^2}
	\bigg[
	2r\left(1+\frac{1}{\sqrt{r^4-1}}\right)
	+
	\frac{2}{r}\left(1+\frac{r^2}{\sqrt{r^4-1}}\right)
	\bigg]
	\|F\|_{\infty,\CA_r}\\
	&= 2\left(1+\frac{2r^2}{(r^2+1)\sqrt{r^4-1}}\right)
	\|F\|_{\infty,\CA_r}\\
	&=K_{dc}^{(1)}\|F\|_{\infty, \CA_r},	
		\end{align*}
where $K_{dc}^{(d)}$ is as in the statement of the theorem.	Fix $z_1 \in \CA_r$ and define $G_{z_1}: \CA_r^{d-1} \to M_n(\C)$ as $G_{z_1}(z_2, \dotsc, z_d)= G(z_1, z_2, \dotsc, z_{d})$. By induction hypothesis, it follows that
		\begin{align*}
			\|F(z_1)\|=\|G(z_1, J_2, \dotsc, J_d)\|
			&=\|G_{z_1}(J_2, \dotsc, J_d)\| \leq K_{dc}^{(d-1)}\|G_{z_1}\|_{\infty, \CA_r^{d-1}} \leq K_{dc}^{(d-1)}\|G\|_{\infty, \CA_r^{d}}
		\end{align*}
		and so, $
		\|G(T_1,\dots,T_d)\| \leq \|F(J_1)\|
		\leq	K_{dc}^{(1)}\|F\|_{\infty, \CA_r} \leq K_{dc}^{(1)}K_{dc}^{(d-1)}
		\|G\|_{\infty,\CA_r^d}=K_{dc}^{(d)}\|G\|_{\infty, \CA_r^d}$.
	\end{proof} 
	
	As an application of the above theorem, we have the following result.
	
	\begin{cor}\label{thm:polyannulusII}
		Let $K_{dc}(\CA_r^d)$ and $K_{dc}^{\mathrm{full}}(\CA_r^d)$ be the smallest constants for which the polyannulus $\CA_r^d$ is a $K_{dc}(\CA_r^d)$-spectral set and a complete $K_{dc}^{\mathrm{full}}(\CA_r^d)$-spectral set, respectively, for every doubly commuting $d$-tuple of operators in $Q\A_r$. Then 
		\[
		2^d \leq K_{dc}(\CA_r^d) \leq K_{dc}^{\text{full}}(\CA_r^d) \leq \left[2\left(1+\frac{2r^2}{(r^2+1)\sqrt{r^4-1}}\right)\right]^d
		\ \ \text{and} \ \
		\underset{r \to \infty}{\lim} \ K_{dc}(\CA_r^d)=\underset{r \to \infty}{\lim}  K_{dc}^{\text{full}}(\CA_r^d)=2^d.
		\]
	\end{cor}
	
	\begin{proof}
		By definition of spectral sets and complete spectral sets, it follows that $K_{dc}(\CA_r^d) \leq K_{dc}^{\text{full}}(\CA_r^d)$. We have by Corollary \ref{cor_402} and Theorem \ref{thm_403} that
		\[
		2^d \leq K_{dc}(\CA_r^d) \leq K_{dc}^{\text{full}}(\CA_r^d) \leq \left[2\left(1+\frac{2r^2}{(r^2+1)\sqrt{r^4-1}}\right)\right]^d.
		\]
Since $\CA_r\subseteq \CA_s$ whenever $1<r\leq s$, the function
$\mathfrak{K}^{\mathrm{full}}:(1,\infty)\to\mathbb R$ given by
$\mathfrak{K}^{\mathrm{full}}(r)=K_{dc}^{\mathrm{full}}(\CA_r^d)
$
is non-increasing and so,
$
\underset{r\to\infty}{\lim}K_{dc}^{\mathrm{full}}(\CA_r^d)
$
exists. Consequently, $\underset{r\to\infty}{\lim}K_{dc}(\CA_r^d)=\underset{r\to\infty}{\lim}K_{dc}^{\mathrm{full}}(\CA_r^d)=
2^d$.
	\end{proof}
	
	We conclude this section with the following similarity theorem for doubly commuting operators in $Q\A_r$. The key ingredient in the proof is Paulsen's similarity theorem (see Theorem 9.1 in \cite{Paulsen}). 
	
	\begin{thm}\label{thm_sim}
		Let $\underline{T}=(T_1, \dotsc, T_d)$ be a doubly commuting tuple of operators in $Q\A_r$ acting on a Hilbert space $\HS$. Then there exists an invertible operator $S$ on $\HS$ satisfying
		\[
		\|S\|\cdot \|S^{-1}\| \leq \left[2\left(1+\frac{2r^2}{(r^2+1)\sqrt{r^4-1}}\right)\right]^d
		\] 
		such that the tuple 
		$
		\underline{R}=(S^{-1}T_1S, \dotsc, S^{-1}T_dS)
		$ 
		has $\CA_r^d$ as a complete spectral set.
	\end{thm}	
	
	\begin{proof}
		Let $\mathfrak{A}=\text{Rat}(\CA_r^d)$, which is a unital subalgebra of the $C^*$-algebra $C(\CA_r^d)$ equipped with the supremum norm $\|.\|_{\infty, \CA_r^d}$. In other words, $\mathfrak{A}$ is an operator algebra. Consider the map $\rho: \mathfrak{A} \to \mathcal{B}(\HS)$ given by $\rho(f)=f(T_1, \dotsc, T_d)$. For $n \in \mathbb{N}$, the $n$-th amplification map of $\rho$ is the map
		\[
		\rho^{(n)}: M_n(\mathfrak{A}) \to M_n(\mathcal{B}(\HS))\quad \text{defined as} \quad \rho^{(n)}([f_{ij}]_{i, j=1}^n)=[\rho(f_{ij})]_{i, j=1}^n.
		\] 
		We say that $\rho$ is completely bounded if $\|\rho\|_{cb}=\sup \{\|\rho^{(n)}\|: n \in \mathbb{N} \}< \infty$. Clearly, $\rho$ is a unital homomorphism. We show that $\rho$ is completely bounded. Let $n \in \mathbb{N}$ and $[f_{ij}]_{i, j=1}^n \in M_n(\mathfrak{A})$. Then
		\[
		\|[\rho(f_{ij})]_{i, j=1}^n\| =\|[f_{ij}(T_1, \dotsc, T_d)]_{i, j=1}^n\| \leq \left[2\left(1+\frac{2r^2}{(r^2+1)\sqrt{r^4-1}}\right)\right]^d \|[f_{ij}]_{i, j=1}^n\|_{\infty, \CA_r^d},
		\] 
		where the last inequality follows from Theorem \ref{thm_403}. Therefore, $\rho$ is a unital, completely bounded homomorphism with 
		\[
		\|\rho\|_{cb}\leq \left[2\left(1+\frac{2r^2}{(r^2+1)\sqrt{r^4-1}}\right)\right]^d.
		\]
		It follows from Theorem 9.1 in \cite{Paulsen}
		that there exists an invertible operator $S$ with $\|S\| \cdot \|S^{-1}\|\leq \|\rho\|_{cb}$ such that the map $\zeta: \mathfrak{A} \to \mathcal{B}(\HS)$ given by $f\mapsto S^{-1}\rho(f)S$ is completely contractive. Let $\pi_j(z_1, \dotsc, z_d)=z_j$ be the $j$-th coordinate map for $1 \leq j \leq d$. Define $\underline{R}=(R_1, \dotsc, R_d)=(\zeta(\pi_1), \dotsc, \zeta(\pi_d))$. Then
		\[
		(R_1, \dotsc, R_d)=(\zeta(\pi_1), \dotsc, \zeta(\pi_d))=(S^{-1}\rho(\pi_1)S, \dotsc, S^{-1}\rho(\pi_d)S)=(S^{-1}T_1S, \dotsc, S^{-1}T_dS).
		\]
		Let $[f_{ij}]_{i, j=1}^n \in M_n(\text{Rat}(\CA_r^d))$. Since $\zeta$ is a completely contractive map, it follows that
		\begin{align*}
			\|[f_{ij}(\underline{R})]_{i, j=1}^n\|
			=\|[f_{ij}(S^{-1}T_1S,  \dotsc, S^{-1}T_d S)]_{i, j=1}^n\|
			&=\|[S^{-1}f_{ij}(T_1, \dotsc, T_d)S]_{i, j=1}^n\|\\
			&=\|[S^{-1}\rho(f_{ij})S]_{i, j=1}^n\|\\
			&=	\|[\zeta(f_{ij})]_{i, j=1}^n\|\\
			& \leq \|[f_{ij}]_{i, j=1}^n\|_{M_n(\mathfrak{A})}\\
			&= \|[f_{ij}]_{i, j=1}^n\|_{\infty, \CA_r^d}.
		\end{align*} 
		By above inequality, it follows that $\|R_j\|, \|R_j^{-1}\|$ and $\sigma(R_j) \subseteq \CA_r$ for $1 \leq j \leq d$. Thus, the Taylor joint spectrum $\sigma_T(R_1, \dotsc, R_d)$ of $(R_1, \dotsc, R_d)$ is a subset of $\CA_r^d$ and so, $\CA_r^d$ is a complete spectral set for $\underline{R}$. The proof is now complete. 
	\end{proof}	
	
\section{Commuting tuples of operators in $Q\A_r$}\label{sec_05}	

\noindent We proved in Section \ref{sec_04} that $\CA_r^d$ is a complete $K_{dc}^{(d)}$-spectral set for doubly commuting $d$-tuples of operators in $Q\A_r$. In this section, we study spectral constant estimates for a subclass of commuting tuples of operators in $Q\A_r$. Needless to mention, one should not expect the doubly commuting arguments to extend to the commutative setting. To begin with, we present the following result.

\begin{prop}\label{prop_501}
	Let $g$ be a holomorphic function on the polyannulus $\CA_r^d$. Then we have a decomposition of $g$ into $2^d$ functions given by
	\begin{align}\label{eqn_g2n} 
		g(z_1, \dotsc, z_d)=\underset{\mu}{\sum} \ g_{\mu}\left(z_1^{\mu(1)}, \dotsc, z_n^{\mu(d)}\right) \qquad \text{on \ $\CA_r^d$},
	\end{align}
	where the sum varies over all functions $\mu: \{1, \dotsc, d\} \to \{1, -1\}$. Moreover, each $g_{\mu}$ is a holomorphic function on the polydisc $(r\overline{\mathbb{D}})^d$ and 
	\begin{align}\label{eqn_g2n_esti}
		\|g_{\mu}\|_{\infty, (r\mathbb{T})^d} \leq
		\left(1+\frac{1}{\sqrt{r^4-1}}\right)^{t}\left(1+\frac{r^2}{\sqrt{r^4-1}}\right)^{d-t}\|g\|_{\infty, \CA_r^d},
	\end{align}
	where $t \in \{0, \dotsc, d\}$ is the cardinality of the set $\mu^{-1}(\{1\})$.
\end{prop}	

\begin{proof}
	Let $g$ be a holomorphic map on $\CA_r^d$. The existence of a decomposition of $g$ into $2^d$ functions follows from Proposition 4.1 in \cite{Pal_Pascoe}. We briefly recall it here. By Laurent series representation,
	\begin{align}\label{eqn_gsplit1}
		g(z_1, \dotsc, z_d)=\overset{\infty}{\underset{\nu_1=-\infty}{\sum}}\dotsc \overset{\infty}{\underset{\nu_d=-\infty}{\sum}}a_{\nu_1,\dotsc, \nu_d}z_1^{\nu_1}\dotsc z_d^{\nu_d}=\underset{\mu}{\sum} \ g_{\mu}\left(z_1^{\mu(1)}, \dotsc, z_d^{\mu(d)}\right), 
	\end{align}
	where the sum varies over all functions $\mu: \{1, \dotsc, d\} \to \{1, -1\}$. The functions $g_{\mu}$ are given by 
\[
g_\mu(w_1, \dotsc, w_d)=\overset{\infty}{\underset{\nu_1=\sigma(1)}{\sum}}\dotsc \overset{\infty}{\underset{\nu_d=\sigma(d)}{\sum}}a_{\mu(1)\nu_1,\dotsc, \mu(d)\nu_d}w_1^{\nu_1}\dotsc w_d^{\nu_d}, \ \  \text{where} \ \ 
\sigma(k)= \left\{
\begin{array}{ll}
	0 & \mbox{ if } \; \mu(k)=1 \\
	1 &  \mbox{ if} \; \mu(k)=-1 \\
\end{array} 
\right.
\]
for $1 \leq k \leq d$. Also, the above series converges uniformly and absolutely on the polydisc $(r\overline{\mathbb{D}})^d$ and each $g_\mu$ defines a holomorphic function on $(r\overline{\mathbb{D}})^d$. It only remains to show the inequality as in \eqref{eqn_g2n_esti}. To make the algorithm clear to the readers, we first discuss the case when $d=1$ and then we prove the desired conclusion for $d=2$. The general induction method follows the same techniques. For a holomorphic function $g$ on $\CA_r$, we can write $g(z)=g_1(z)+g_2(1\slash z)$, where $g_1(z)=\overset{\infty}{\underset{n=0}{\sum}}a_nz^n$ and $g_2(z)=\overset{\infty}{\underset{n=1}{\sum}}a_{-n}z^n$. Let $P_r=1+(1\slash \sqrt{r^4-1})$ and $Q_r=1+(r^2\slash \sqrt{r^4-1})$.
For $\xi \in r\T$, an application of Parseval's identity as in Theorem \ref{thm_206} gives that
\begin{align*}
	\left|g_1(\xi)\right|
	&=\left|g(\xi)-\overset{\infty}{\underset{n=1}{\sum}}\frac{a_{-n}}{\xi^n}\right| 
	\leq \|g\|_{\infty, \CA_r}+\left(\overset{\infty}{\underset{n=1}{\sum}}|a_{-n}|^2r^{2n}\right)^{1\slash 2}\left(\overset{\infty}{\underset{n=1}{\sum}}\frac{1}{r^{4n}}\right)^{1\slash 2} \leq  P_r \|g\|_{\infty, \CA_r} \ \text{and} \\
	 \left|g_2(\xi)\right|&=\left|g(1\slash \xi)-\overset{\infty}{\underset{n=0}{\sum}}\frac{a_{n}}{\xi^n}\right| 
	\leq 
	\|g\|_{\infty, \CA_r}+\left(\overset{\infty}{\underset{n=0}{\sum}}|a_{n}|^2r^{2n}\right)^{1\slash 2}\left(\overset{\infty}{\underset{n=0}{\sum}}\frac{1}{r^{4n}}\right)^{1\slash 2}\leq   Q_r\|g\|_{\infty, \CA_r}. 
\end{align*}
This establishes the conclusion for $d=1$. Next, consider the case $d=2$. Let $g$ be a holomorphic function on $\CA_r^2$. Then $g(z_1, z_2)
=g_1(z_1, z_2)+g_2(z_1, z_2^{-1})+g_3(z_1^{-1}, z_2)+g_4(z_1^{-1}, z_2^{-1})$, where 
\begin{align*}
	& g_1(z_1, z_2)=\overset{\infty}{\underset{n=0}{\sum}}\ \overset{\infty}{\underset{m=0}{\sum}}a_{n,m} \ z_1^nz_2^m, \quad \qquad g_2(z_1, w_2)=\overset{\infty}{\underset{n=0}{\sum}}\ \overset{\infty}{\underset{m=1}{\sum}}a_{n,-m}z_1^nw_2^m, \notag \\
	& g_3(w_1, z_2)=\overset{\infty}{\underset{n=1}{\sum}}\ \overset{\infty}{\underset{m=0}{\sum}}a_{-n,m} w_1^nz_2^m, \qquad g_4(w_1, w_2)=\overset{\infty}{\underset{n=1}{\sum}}\ \overset{\infty}{\underset{m=1}{\sum}}a_{-n,-m}w_1^nw_2^m.
\end{align*}
We mention here that the main idea in the case $d=2$ is similar to that of Proposition 3.1 in \cite{Pal_Pascoe}. Therefore, we only describe the modifications needed for the desired estimates. First, we estimate $\|g_1\|_{\infty, r\mathbb{T} \times r\mathbb{T}}$. Let $\xi_1, \xi_2 \in r\T$ and define
\[
b_{-n}=\overset{\infty}{\underset{m=-\infty}{\sum}}a_{-n,m} \xi_2^m \quad \text{and} \quad c_{-m}=\overset{\infty}{\underset{n=-\infty}{\sum}}a_{n, -m} \xi_1^n \quad (n, m \in \mathbb{Z}).
\] 
We have by Equation (12) in \cite{Pal_Pascoe} that 
\begin{align*}
	\left|\overset{\infty}{\underset{n=0}{\sum}} \ \overset{\infty}{\underset{m=-\infty}{\sum}}a_{n,m} \xi_1^n \xi_2^m\right| = \left|g(\xi_1, \xi_2)-\overset{\infty}{\underset{n=1}{\sum}} \ \overset{\infty}{\underset{m=-\infty}{\sum}}a_{-n,m} \frac{\xi_2^m}{\xi_1^n}\right| \leq \|g\|_{\infty, \CA_r^2}\left(1+ \frac{1}{\sqrt{r^4-1}}\right).
\end{align*}
A straightforward calculation gives that
\begin{align*}
	& \quad \left|	\overset{\infty}{\underset{n=0}{\sum}} \ \overset{\infty}{\underset{m=1}{\sum}}\frac{a_{n,-m}}{\xi_2^m}\xi_1^n
	\right|\\
	&=\left|\overset{\infty}{\underset{n=-\infty}{\sum}} \ \overset{\infty}{\underset{m=1}{\sum}}\frac{a_{n,-m}}{\xi_2^m}\xi_1^n-\overset{-1}{\underset{n=-\infty}{\sum}} \ \overset{\infty}{\underset{m=1}{\sum}}\frac{a_{n,-m}}{\xi_2^m}\xi_1^n
	\right| \notag \\
	& \leq 	\sqrt{\left(\overset{\infty}{\underset{m=1}{\sum}}|c_{-m}|^2r^{2m}\right)\left(\overset{\infty}{\underset{m=1}{\sum}}\frac{1}{r^{4m}}\right)}
	+
	\sqrt{\left(\overset{\infty}{\underset{n, m=1}{\sum}}|a_{-n, -m}|^2r^{2n+2m}\right)\left(\overset{\infty}{\underset{n, m=1}{\sum}}\frac{1}{r^{4n+4m}}\right)} \notag \\
	&\leq \|g\|_{\infty, \CA_r^2}\left[\frac{1}{\sqrt{r^4-1}}+\left(\frac{1}{\sqrt{r^4-1}}\right)\left(\frac{1}{\sqrt{r^4-1}}\right)\right] \quad \quad [\text{by Parseval's identity}].	
\end{align*}
Consequently,
\begin{align}\label{eqn_0308}
\|g_1\|_{\infty, r\mathbb{T} \times r\mathbb{T}}=\sup_{\xi_1, \xi_2 \in r\T}\left| \overset{\infty}{\underset{n=0}{\sum}}\ \overset{\infty}{\underset{m=-\infty}{\sum}}a_{n,m} \xi_1^n\xi_2^m-\overset{\infty}{\underset{n=0}{\sum}}  \ \overset{\infty}{\underset{m=1}{\sum}}\frac{a_{n,-m}}{\xi_2^m} \xi_1^n\right| 
	\leq P_r^2\|g\|_{\infty, \CA_r^2}.
\end{align}
Next, we provide an estimate for $\|g_2\|_{\infty, r\mathbb{T} \times r\mathbb{T}}$. Let $\xi_1, \xi_2 \in r\T$  and set
\[
d_{-n}=\overset{\infty}{\underset{m=-\infty}{\sum}}a_{-n,m}\xi_2^{-m} \quad \text{and} \quad e_{-m}=\overset{\infty}{\underset{n=-\infty}{\sum}}a_{n, m} \xi_1^n \quad (n, m \in \mathbb{Z}).
\]  
We have by Equation (17) in \cite{Pal_Pascoe} that
\begin{align*}
	\left|\overset{\infty}{\underset{n=0}{\sum}} \  \overset{\infty}{\underset{m=-\infty}{\sum}}a_{n,-m} \xi_2^{m} \xi_1^n\right|
	\leq \|g\|_{\infty, \CA_r^2}\left(1+ \frac{1}{\sqrt{r^4-1}}\right).
\end{align*}
Also, we have
\begin{align*}
	& \quad \left|	\overset{\infty}{\underset{n=0}{\sum}} \ \overset{\infty}{\underset{m=0}{\sum}}\frac{a_{n,m}\xi_1^n}{\xi_2^m}
	\right|\\
	&=\left|\overset{\infty}{\underset{m=0}{\sum}} \ \overset{\infty}{\underset{n=-\infty}{\sum}}\frac{a_{n,m}\xi_1^n}{\xi_2^m}-\overset{\infty}{\underset{m=0}{\sum}} \ \overset{-1}{\underset{n=-\infty}{\sum}}\frac{a_{n,m}\xi_1^n}{\xi_2^m}
	\right| \notag  \\
	& \leq 	\sqrt{\left(\overset{\infty}{\underset{m=0}{\sum}}|e_{-m}|^2r^{2m}\right)\left(\overset{\infty}{\underset{m=0}{\sum}}\frac{1}{r^{4m}}\right)} 	+
	\sqrt{\left(\sum_{\substack{n=1\\ m=0}}^{\infty}|a_{-n, m}|^2r^{2n+2m}\right)\left(\sum_{\substack{n=1\\ m=0}}^{\infty}\frac{1}{r^{4n+4m}}\right)} \notag \\
	&\leq \|g\|_{\infty, \CA_r^2}\left[\frac{r^2}{\sqrt{r^4-1}}+\left(\frac{r^2}{\sqrt{r^4-1}}\right)\left(\frac{1}{\sqrt{r^4-1}}\right)\right] \quad [\text{by Parseval's identity}].
\end{align*}
Therefore,
\begin{align}\label{eqn_0315}
\|g_2\|_{\infty, r\mathbb{T} \times r\mathbb{T}}=\sup_{\xi_1, \xi_2 \in r\T}\left| \overset{\infty}{\underset{n=0}{\sum}}\ \overset{\infty}{\underset{m=-\infty}{\sum}}a_{n, -m} \xi_1^n\xi_2^m-\overset{\infty}{\underset{n=0}{\sum}}   \ \overset{\infty}{\underset{m=0}{\sum}}\frac{a_{n,m}}{\xi_2^m} \xi_1^n\right| 
	\leq P_rQ_r\|g\|_{\infty, \CA_r^2}.
\end{align}
We now present an estimate for $\|g_3\|_{\infty, r\mathbb{T} \times r\mathbb{T}}$.  Let $\xi_1, \xi_2 \in r\T$ and take 
\[
\omega_{-n}=\overset{\infty}{\underset{m=-\infty}{\sum}}a_{n,m} \xi_2^m \quad \text{and} \quad \gamma_{-m}=\overset{\infty}{\underset{n=-\infty}{\sum}}a_{n, m}\xi_1^{-n} \quad (n, m \in \mathbb{Z}).
\] 
It follows from Equation $(21)$ in \cite{Pal_Pascoe} that
\begin{align*}
	\left|\overset{\infty}{\underset{m=0}{\sum}} \  \overset{\infty}{\underset{n=-\infty}{\sum}}a_{-n, m} \xi_1^{n} \xi_2^m\right|
\leq \|g\|_{\infty, \CA_r^2}\left(1+ \frac{1}{\sqrt{r^4-1}}\right).
\end{align*}
A routine computation gives that
\begin{align*}
	& \quad \left|	\overset{\infty}{\underset{n=0}{\sum}} \ \overset{\infty}{\underset{m=0}{\sum}}\frac{a_{n,m}\xi_2^m}{\xi_1^n}
	\right|\\
	& =\left|\overset{\infty}{\underset{n=0}{\sum}} \ \overset{\infty}{\underset{m=-\infty}{\sum}}\frac{a_{n,m}\xi_2^m}{\xi_1^n}-\overset{\infty}{\underset{n=0}{\sum}} \ \overset{\infty}{\underset{m=1}{\sum}}\frac{a_{n,-m}}{\xi_1^n\xi_2^m}
	\right| \notag  \\
	& \leq 	\sqrt{\left(\overset{\infty}{\underset{n=0}{\sum}}|\omega_{-n}|^2r^{2n}\right)\left(\overset{\infty}{\underset{n=0}{\sum}}\frac{1}{r^{4n}}\right)} 	+
	\sqrt{\left(\sum_{\substack{n=0\\ m=1}}^{\infty}|a_{n, -m}|^2r^{2n+2m}\right)\left(\sum_{\substack{n=0\\ m=1}}^{\infty}\frac{1}{r^{4n+4m}}\right)} \notag \\
		&\leq \|g\|_{\infty, \CA_r^2}\left[\frac{r^2}{\sqrt{r^4-1}}+\left(\frac{r^2}{\sqrt{r^4-1}}\right)\left(\frac{1}{\sqrt{r^4-1}}\right)\right] \quad [\text{by Parseval's identity}].
\end{align*}
Combining the above last two inequalities, we have
\begin{align}\label{eqn_3016}
	\|g_3\|_{\infty, r\mathbb{T} \times r\mathbb{T}}=\sup_{\xi_1, \xi_2 \in r\T}\left| \overset{\infty}{\underset{m=0}{\sum}} \  \overset{\infty}{\underset{n=-\infty}{\sum}}a_{-n, m} \xi_1^{n} \xi_2^m-\overset{\infty}{\underset{n=0}{\sum}} \ \overset{\infty}{\underset{m=0}{\sum}}\frac{a_{n,m}\xi_2^m}{\xi_1^n}\right| \leq P_rQ_r\|g\|_{\infty, \CA_r^2}.
\end{align}
Finally, we prove an estimate for $\|g_4\|_{\infty, r\mathbb{T} \times r\mathbb{T}}$. Let $\xi_1, \xi_2 \in r\T$ and let
\[
\nu_{-n}=\overset{\infty}{\underset{m=-\infty}{\sum}}a_{n,m}\xi_2^{-m} \quad \text{and} \quad \mu_{-m}=\overset{\infty}{\underset{n=-\infty}{\sum}}a_{-n, m} \xi_1^n \quad (n, m \in \mathbb{Z}).
\]  
We have by Equation (25) in \cite{Pal_Pascoe} that
\begin{align*}
	\left|\overset{\infty}{\underset{n=1}{\sum}} \  \overset{\infty}{\underset{m=-\infty}{\sum}}a_{-n, -m} \xi_1^{n} \xi_2^m\right|
\leq \|g\|_{\infty, \CA_r^2}\left(1+ \frac{r^2}{\sqrt{r^4-1}}\right).
\end{align*}
Following the similar arguments as above for estimating the previous norm entities, we have
\begin{align*}
	 & \quad \left|	\overset{\infty}{\underset{n=1}{\sum}} \ \overset{\infty}{\underset{m=0}{\sum}}\frac{a_{-n,m}\xi_1^n}{\xi_2^m}
	\right|\\
	& \leq \left|\overset{\infty}{\underset{m=0}{\sum}} \frac{\mu_{-m}}{\xi_2^m}\right|+\left|\overset{\infty}{\underset{n=0}{\sum}} \ \overset{\infty}{\underset{m=0}{\sum}}\frac{a_{n, -m}}{\xi_1^n\xi_2^m}
	\right| \notag \\
	& \leq 	\sqrt{\left(\overset{\infty}{\underset{m=0}{\sum}}|\mu_{-m}|^2r^{2m}\right)\left(\overset{\infty}{\underset{m=0}{\sum}}\frac{1}{r^{4m}}\right)} +
	\sqrt{\left(\sum_{\substack{n, m=0}}^{\infty}|a_{n, -m}|^2r^{2n+2m}\right)\left(\sum_{\substack{n, m=0}}^{\infty}\frac{1}{r^{4n+4m}}\right)} \notag \\
&\leq \|g\|_{\infty, \CA_r^2}\left[\frac{r^2}{\sqrt{r^4-1}}+\left(\frac{r^2}{\sqrt{r^4-1}}\right)\left(\frac{r^2}{\sqrt{r^4-1}}\right)\right] \quad \quad [\text{by Parseval's identity}].	
\end{align*}
Thus, 
\begin{align}\label{eqn_3020}
	\|g_4\|_{\infty, r\mathbb{T} \times r\mathbb{T}} \leq \sup_{\xi_1, \xi_2 \in r\T} \left|\overset{\infty}{\underset{n=1}{\sum}} \  \overset{\infty}{\underset{m=-\infty}{\sum}}a_{-n, -m} \xi_1^{n} \xi_2^m-\overset{\infty}{\underset{n=1}{\sum}} \ \overset{\infty}{\underset{m=0}{\sum}}\frac{a_{-n,m}\xi_1^n}{\xi_2^m}\right|
	\leq Q_r^2\|g\|_{\infty, \CA_r^2}.
\end{align}
Combining \eqref{eqn_0308}, \eqref{eqn_0315}, \eqref{eqn_3016} and \eqref{eqn_3020}, the desired conclusion follows.
\end{proof}

Proposition \ref{prop_501} provides an explicit upper bound on the supremum norm $\|.\|_{\infty, (r\T)^d}$ of each component $g_\mu$ as in the decomposition \eqref{eqn_g2n}. We mention here that Proposition 4.1 of \cite{Pal_Pascoe} yields
\[
\|g_{\mu}\|_{\infty,(r\T)^d}
\leq
\left(\frac{r^2}{r^2-1}\right)^t
\left(\frac{2r^2-1}{r^2-1}\right)^{d-t}
\|g\|_{\infty,\CA_r^d},
\]
where $t$ is the cardinality of the set $\mu^{-1}(\{1\})$. Note that
\[
\left(1+\frac{1}{\sqrt{r^4-1}}\right) 
\lneq 
\frac{r^2}{r^2-1}
\qquad\text{and}\qquad
\left(1+\frac{r^2}{\sqrt{r^4-1}}\right) 
\lneq 
\frac{2r^2-1}{r^2-1}
\]
for every $r>1$ and so, the estimate in \eqref{eqn_g2n_esti} improves the corresponding bounds from \cite{Pal_Pascoe}. Let us denote by $Q\A_r^{\text{von}}$ the class of all commuting tuples $(T_1, \dotsc, T_d)$ of operators in $Q\A_r$ satisfying the hypothesis of Theorem \ref{thm_502}, namely,
\[
\|f(T_1^{\mu(1)}, \dotsc, T_d^{\mu(d)})\| \leq \|f\|_{\infty, (r\overline{\D})^d}
\]
for every holomorphic function $f$ on $(r\overline{\D})^d$ and every map $\mu: \{1, \dotsc, d\} \to \{1, -1\}$. An application of Proposition \ref{prop_501} gives that $\CA_r^d$ is a $K$-spectral set for tuples in $Q\A_r^{\text{von}}$. Needless to mention, the class $Q\A_r^{\text{von}}$ is non-empty. For example, doubly commuting tuples of operators in $Q\A_r$ belong to $Q\A_r^{\text{von}}$. Holbrook \cite{Holbrook} proved the success of unitary dilation of commuting contractions acting on $\C^2$. Consequently, a re-scaling argument gives that commuting tuples of operators in $Q\A_r$ acting on $\C^2$ belong to $Q\A_r^{\text{von}}$. We conclude this section with the following result.
	
\begin{thm}\label{thm_502}
$\CA_r^d$ is a $K(r, d)$-spectral for commuting tuples of operators in $Q\A_r^{\text{von}}$, where
	\[
	K(r, d)=\left(2+\frac{1+r^2}{\sqrt{r^4-1}}\right)^d.
	\]
\end{thm}	
	
\begin{proof}
Let $(T_1, \dotsc, T_d)$ be a commuting tuple of operators in $Q\A_r^{\text{von}}$, and let $\mathfrak{M}_d$ be the class of all functions $\mu: \{1, \dotsc, d\} \to \{1, -1\}$. Then $(T_1^{\mu(1)}, \dotsc, T_d^{\mu(d)})$ satisfies the von Neumann's inequality on $(r\overline{\D})^d$ for all holomorphic functions $f$ on $(r\overline{\D})^d$. For $\mu \in \mathfrak{M}_d$, let $|\mu^{-1}(\{1\})|$ be the cardinality of $\mu^{-1}\{1\}$. We have by Proposition \ref{prop_501} that
	\begin{align*} 
		g(z_1, \dotsc, z_d)=\underset{\mu \in \mathfrak{M}_d}{\sum} \ g_{\mu}\left(z_1^{\mu(1)}, \dotsc, z_d^{\mu(d)}\right) \quad \text{on \ $\CA_r^d$}
	\end{align*}
	such that each $g_{\mu}$ is holomorphic on $(r\overline{\mathbb{D}})^d$. Let $\mu \in \mathfrak{M}_d$. Since each $T_j$ is in $Q\A_r$, it follows that $\|T_j\|, \|T_j^{-1}\| \leq r$ and so, $\sigma_T(T_1^{\mu(1)}, \dotsc, T_d^{\mu(d)}) \subseteq r(\overline{\D})^d$ and so, $g(T_1^{\mu(1)}, \dotsc, T_d^{\mu(d)})$ is a well-defined operator. Consequently, we have that
		\begin{align*}
		& \quad \|g(T_1, \dotsc, T_d)\|\\ 
		& \leq \underset{\mu \in \mathfrak{M}_d}{\sum} \left\|g_{\mu}\left(T_1^{\mu(1)}, \dotsc, T_d^{\mu(d)}\right)\right\| \\
			& \leq \underset{\mu \in \mathfrak{M}_d}{\sum} \left\|g_{\mu}\left(z_1^{\mu(1)}, \dotsc, z_d^{\mu(d)}\right)\right\|_{\infty, (r\T)^d} \quad [\text{by given hypothesis}] \\
		& \leq \left[\underset{\mu \in \mathfrak{M}_d}{\sum} \left(1+\frac{1}{\sqrt{r^4-1}}\right)^{|\mu^{-1}(\{1\})|}\left(1+\frac{r^2}{\sqrt{r^4-1}}\right)^{d-|\mu^{-1}(\{1\})|}\right]\|g\|_{\infty, \CA_r^d}  \quad \text{[by Proposition \ref{prop_501}]}\\
	&=\left[\overset{n}{\underset{t=0}{\sum}}\binom{d}{t}\left(1+\frac{1}{\sqrt{r^4-1}}\right)^{t}\left(1+\frac{r^2}{\sqrt{r^4-1}}\right)^{d-t}\right]\|g\|_{\infty, \CA_r^d}\\
		&=\left(2+\frac{1+r^2}{\sqrt{r^4-1}}\right)^d\|g\|_{\infty, \CA_r^d},
	\end{align*}	
which completes the proof.	
\end{proof}	
	
The case $d=2$	in Theorem \ref{thm_502} corresponds to a commuting pair $(T_1, T_2)$ of operators in $Q\A_r$. For every function $\mu:\{1, 2\} \to \{1, -1\}$, it follows that $(r^{-1}T_1^{\mu(1)}, r^{-1}T_2^{\mu(2)})$ is a commuting pair of contractions and by Ando's dilation theorem (see \cite{Ando}), each pair $(T_1^{\mu(1)}, T_2^{\mu(2)})$ satisfies von Neumann's inequality on $r\overline{\D} \times r\overline{\D}$. Consequently, we have by Theorem \ref{thm_502} that
	\[
	\|g(T_1, T_2)\| \leq \left(2+\frac{1+r^2}{\sqrt{r^4-1}}\right)^2\|g\|_{\infty, \CA_r^2}
	\]
for all $g \in \text{Rat}(\CA_r^2)$ and every commuting pair $(T_1, T_2)$ of operators in $Q\A_r$. This improves Theorem 3.2 of \cite{Pal_Pascoe}, where it was shown that for every $g \in \text{Rat}(\CA_r^2)$,
	\[
\|g(T_1, T_2)\| \leq \left[4+4\left(\frac{r^2+1}{r^2-1}\right)^{1\slash 2}+\left(\frac{r^2+1}{r^2-1}\right)^2 \ \right] \|g\|_{\infty, \CA_r^2}. 
\]

\medskip

	\noindent \textbf{Acknowledgment.} The author is supported through the IIT Bombay RDF Grant of Prof. Sourav Pal with Project Code RI/0115-10001427. The author is grateful to Prof. Sourav Pal for carefully reading the manuscript and for several valuable suggestions.

\end{document}